\documentclass[11pt]{article}
 \pdfoutput=1                           
\usepackage{a4wide}

\usepackage{amssymb,amsthm,amsmath,amsfonts}
\usepackage{graphicx}

\if10     
\usepackage[mathlines]{lineno}
\newcommand*\patchAmsMathEnvironmentForLineno[1]{%
  \expandafter\let\csname old#1\expandafter\endcsname\csname #1\endcsname
  \expandafter\let\csname oldend#1\expandafter\endcsname\csname end#1\endcsname
  \renewenvironment{#1}%
     {\linenomath\csname old#1\endcsname}%
     {\csname oldend#1\endcsname\endlinenomath}}%
\newcommand*\patchBothAmsMathEnvironmentsForLineno[1]{%
  \patchAmsMathEnvironmentForLineno{#1}%
  \patchAmsMathEnvironmentForLineno{#1*}}%
\AtBeginDocument{%
\patchBothAmsMathEnvironmentsForLineno{equation}%
\patchBothAmsMathEnvironmentsForLineno{align}%
\patchBothAmsMathEnvironmentsForLineno{flalign}%
\patchBothAmsMathEnvironmentsForLineno{alignat}%
\patchBothAmsMathEnvironmentsForLineno{gather}%
\patchBothAmsMathEnvironmentsForLineno{multline}%
}
\linenumbers
\fi

\newtheorem{theorem}{Theorem}[section]

\newtheorem{lemma}[theorem]{Lemma}
\newtheorem{proposition}[theorem]{Proposition}

\newtheorem{observation}[theorem]{Observation}

\theoremstyle{definition}

\newtheorem{definition}[theorem]{Definition}
\newtheorem{problem}[theorem]{Problem}

\newcommand{\Av}{\mathrm{Av}}
\newcommand{\cC}{\mathcal{C}}
\newcommand{\C}[2]{\cC_{#1,#2}}
\newcommand{\Cplus}{\cC^{\times}_{1,1}}

\newcommand{\btau}{\overline\tau}
\newcommand{\xt}{X^\text{T}}
\newcommand{\xf}{X^\text{F}}
\newcommand{\xp}{X^+}
\newcommand{\xm}{X^-}
\newcommand{\mo}[1]{\widetilde{#1}}
\newcommand{\moQ}{\mo{Q}}
\newcommand{\moP}{\mo{P}}
\newcommand{\moxt}{\mo{X}^\text{T}}
\newcommand{\moxf}{\mo{X}^\text{F}}
\newcommand{\moxp}{\mo{X}^+}
\newcommand{\moxm}{\mo{X}^-}

\newcommand{\prob}[3]{
    \begin{center}
      \fbox{
        \parbox{0.95\textwidth}{
          #1\\
          \begin{tabular}{rp{0.73\textwidth}}
            \textit{Instance:\ } & #2\\
            \textit{Question:\ } & #3
          \end{tabular}
        }
      }
    \end{center}
}

\DeclareMathOperator{\run}{run}

\usepackage{color}
\usepackage{hyperref}

\definecolor{modra3}{rgb}{.1,.0,.4}
\hypersetup{colorlinks=true, linkcolor=modra3, urlcolor=modra3, citecolor=modra3, pdfpagemode=UseNone, pdfstartview=}

\definecolor{modra}{rgb}{0,0,.8}

\def\inst#1{$^{#1}$}

\begin{document}

\title{Hardness of Permutation Pattern Matching\thanks{Supported by project 
16-01602Y of the Czech Science Foundation.}}

\author{V\'{\i}t Jel\'{\i}nek\inst{1}
\and
Jan Kyn\v{c}l\inst{2}}

\date{}

\maketitle

\begin{center} 
{\footnotesize
\inst{1} Computer Science Institute\\ 
Charles University, Faculty of Mathematics and Physics,\\
Malostransk\'e n\'am.~25, 118 00~ Praha 1, Czech Republic;\\
\texttt{jelinek@iuuk.mff.cuni.cz}
\\\ \\
\inst{2}
Department of Applied Mathematics and Institute for Theoretical Computer Science, \\
Charles University, Faculty of Mathematics and Physics, \\
Malostransk\'e n\'am.~25, 118 00~ Praha 1, Czech Republic; \\
\texttt{kyncl@kam.mff.cuni.cz}
}
\end{center}  

\begin{abstract}
Permutation Pattern Matching (or PPM) is a decision problem whose input is a 
pair of permutations $\pi$ and $\tau$, represented as sequences of integers, 
and the task is to determine whether $\tau$ contains a subsequence 
order-isomorphic to~$\pi$. Bose, Buss and Lubiw proved that PPM is NP-complete 
on general inputs.

We show that PPM is NP-complete even when $\pi$ has no decreasing subsequence of
length~3 and $\tau$ has no decreasing subsequence of length~4. This provides the first 
known example of PPM being hard when one or both of $\pi$ and $\sigma$ are 
restricted to a proper hereditary class of permutations.

This hardness result is tight in the sense that PPM is known to be polynomial when 
both $\pi$ and $\tau$ avoid a decreasing subsequence of length~3, as well as 
when $\pi$ avoids a decreasing subsequence of length~2. The result is also 
tight in another sense: we will show that for any hereditary proper subclass $\cC$ of the class of permutations 
avoiding a decreasing sequence of length~3, there is a polynomial algorithm 
solving PPM instances where $\pi$ is from $\cC$ and $\tau$ is arbitrary.

We also obtain analogous hardness and tractability results for the class of 
so-called skew-merged patterns. 

From these results, we deduce a complexity dichotomy for the PPM problem 
restricted to $\pi$ belonging to $\Av(\rho)$, where $\Av(\rho)$ denotes the 
class of permutations avoiding a permutation~$\rho$. Specifically, we show that 
the problem is polynomial when $\rho$ is in the set $\{1, 12, 21, 132, 213, 
231, 312\}$, and it is NP-complete for any other~$\rho$.
\end{abstract}

\section{Introduction}

A \emph{permutation} of size $n$ is a bijection from the set $\{1,2,\dotsc,n\}$ 
to itself. We represent a permutation $\pi$ of size $n$ by the sequence 
$\pi=\pi(1),\pi(2),\dotsc,\pi(n)$. We let $S_n$ denote the set of permutations 
of size~$n$. When writing out short permutations explicitly, we usually omit 
the punctuation and write, e.g., $312$ instead of $3, 1, 2$. We let $[n]$ 
denote the set $\{1,2,\dotsc,n\}$.

We say that a permutation $\tau=\tau(1),\tau(2),\dotsc,\tau(n)\in S_n$ 
\emph{contains} 
a permutation $\pi=\pi(1),\pi(2),\dotsc,\pi(k)\in S_k$, which we denote by 
$\pi\leq 
\tau$, if $\tau$ has a subsequence $\tau(i_1),\allowbreak 
\tau(i_2),\dotsc,\allowbreak\tau(i_k)$ 
whose elements have the same relative order as the elements of $\pi$, that is, 
for any $a,b\in[k]$ we have $\tau(i_a)<\tau(i_b)$ if and only if 
$\pi(a)<\pi(b)$. We call such a subsequence $\tau(i_1), 
\tau(i_2),\dotsc,\tau(i_k)$ an \emph{occurrence} of $\pi$ in~$\tau$.
If $\tau$ does not contain $\pi$, we say that $\tau$ 
\emph{avoids}~$\pi$ or that $\tau$ is \emph{$\pi$-avoiding}.

In this paper, we study the computational complexity of determining for a given 
pair of permutations $\pi$  and $\tau$ whether $\tau$ contains~$\pi$. In the 
literature, this problem is known as Permutation Pattern Matching, or PPM.

\prob{Permutation Pattern Matching (PPM)}{Permutations $\pi\in S_k$ 
and $\tau\in S_n$.}{Does $\tau$ contain $\pi$?}

In the context of PPM, the permutation $\pi$ is usually called the 
\emph{pattern}, and $\tau$ is called the \emph{text}. When dealing with 
instances of PPM, we always assume that the pattern is at most as long as the 
text, that is, $k\le n$.
 
Observe that PPM can be solved by a simple brute-force algorithm in time
$O(\binom{n}{k}\cdot k)$. Thus, if the pattern $\pi$ were fixed, rather than 
being part of the input, the problem would trivially be polynomial-time 
solvable.

Bose, Buss and Lubiw~\cite{BBL} have shown that PPM is NP-complete. This 
general hardness result has motivated the study of parameterized and restricted 
variants of PPM. 

Guillemot and Marx~\cite{GM} have shown that PPM can be solved in time 
$2^{O(k^2\log k)}\cdot n$; in particular, PPM is fixed-parameter tractable with 
$k$ considered as parameter. The complexity of the Guillemot--Marx algorithm 
was 
later reduced to $2^{O(k^2)}\cdot n$ by Fox~\cite{Fox}.

A different parameterization was considered by Bruner and 
Lackner~\cite{BL_runs}, who proved that PPM can be solved in time 
$O(1.79^{\run(\tau)}\cdot kn)$, where $\text{run}(\tau)$ is the number of 
increasing and decreasing runs in~$\tau$; here an \emph{increasing run} in 
$\tau(1),\tau(2),\dotsc,\tau(n)$ is a maximal consecutive increasing subsequence 
of 
length at least $2$, and decreasing runs are defined analogously. The 
Bruner--Lackner algorithm shows that PPM is fixed-parameter tractable with 
respect to the parameter~$\text{run}(\tau)$. Moreover, since $\run(\tau)\le n$ 
for any permutation $\tau\in S_n$, the Bruner--Lackner algorithm solves PPM in 
time $1.79^n\cdot n^{O(1)}$. It is the first algorithm to improve upon the 
bound 
of $2^n\cdot n^{\Theta(1)}$ achieved by the straightforward brute-force 
approach.

Another algorithm for PPM was described by Albert, Aldred, Atkinson and 
Holton~\cite{AAAH}, and later a similar approach was analyzed by Ahal 
and Rabinovich~\cite{AR08_subpattern}. Ahal and Rabinovich have proved that PPM 
can be solved in time $n^{O(\mathrm{tw}(G_{\pi}))}$, where $G_\pi$ is a graph 
that can be associated to the pattern $\pi$, and $\mathrm{tw}(G_\pi)$ denotes 
the treewidth of~$G_\pi$. We shall give the precise definitions and 
the necessary details of this approach in Section~\ref{sec-poly}.

Apart from the above-mentioned algorithms, which all solve general PPM 
instances, various researchers have obtained efficient solutions for instances 
of PPM where $\pi$ or $\tau$ satisfy some additional restrictions. The most 
natural way to formalize such restrictions is to use the concept of permutation 
class, which we now introduce.

A \emph{permutation class} is a set $\cC$ of permutations with the property 
that for every $\sigma \in \cC$, all the permutations contained in $\sigma$ 
belong to $\cC$ as well.

It is often convenient to describe a permutation class $\cC$ by specifying the 
minimal permutations not belonging to~$\cC$. For a set of permutations $F$, we 
let $\Av(F)$ denote the class of permutations that avoid all the permutations 
in~$F$. Note that for any permutation class $\cC$ there is a unique (possibly 
infinite) antichain of permutations $F$ such that $\cC=\Av(F)$. The set $F$ is 
the \emph{basis} of~$\cC$. 

A \emph{principal permutation class} is a permutation class whose basis has a 
single element. A \emph{proper permutation class} is a permutation class whose 
basis is nonempty, or in other words, a permutation class that does not contain 
all permutations. For a recent overview of the structural theory of permutation 
classes, we refer the interested reader to the survey by Vatter~\cite{Vatter}.

When dealing with specific sets $F$, we often omit nested braces and write, 
e.g., $\Av(321)$ or $\Av(2413, 3142)$, instead of $\Av(\{321\})$ or 
$\Av(\{2413,3142\})$, respectively.

In this paper, we focus on the complexity of PPM when one or both of the inputs 
are restricted to a particular proper permutation class. Following the 
terminology of Albert et al.~\cite{ALLV_321}, we consider, for a permutation 
class~$\cC$, these two restricted versions of PPM:

\prob{$\cC$-Permutation Pattern Matching ($\cC$-PPM)}{A pattern $\pi\in \cC$ of 
size $k$ and a text $\tau\in \cC$ of size~$n$.}{Does $\tau$ contain $\pi$?}

\prob{$\cC$-Pattern Permutation Pattern Matching ($\cC$-Pattern PPM)}{A pattern 
$\pi\in \cC$ of size $k$ and a text $\tau\in S_n$.}{Does $\tau$ contain $\pi$?}

Clearly, any instance of $\cC$-PPM is also an instance of $\cC$-Pattern PPM, 
and in particular, $\cC$-PPM is at most as hard as $\cC$-Pattern PPM. 

Bose, Buss and Lubiw~\cite{BBL} have shown that $\cC$-Pattern PPM is 
polynomially tractable when $\cC$ is the class $\Av(2413,3142)$ of the 
so-called \emph{separable permutations}. Other algorithms for 
$\Av(2413,3142)$-Pattern PPM were given by Ibarra~\cite{Ibarra}, by Albert 
et al.~\cite{AAAH}, and by Yugandhar and Saxena~\cite{YS}.

An even more restricted case of $\cC$-Pattern PPM deals with monotone 
increasing patterns, that is, $\cC=\Av(21)$. In this case, $\cC$-Pattern PPM 
reduces to finding the longest increasing subsequence in a given text. This is 
an old algorithmic problem~\cite{Schensted}, and can be solved in time 
$O(n\log\log n)$~\cite{ChW,Makinen}.

A natural generalization is to consider instances of PPM where patterns and 
texts can be partitioned into a bounded number of monotone sequences. For 
integers $r,s\ge 0$, we say that a permutation $\pi$ is an 
\emph{$(r,s)$-permutation} if $\pi$ can be partitioned into $r$ increasing and 
$s$ decreasing (possibly empty) subsequences. We let $\C{r}{s}$ denote the 
class 
of all $(r,s)$-permutations. In particular, $\Av(21)=\C{1}{0}$, and more 
generally, $\Av(k(k-1)\dots 1)=\C{k-1}{0}$. The $(1,1)$-permutations are also 
known as \emph{skew-merged permutations}, and it is not hard to see that 
$\C{1}{1}=\Av(2143,3412)$; see Atkinson~\cite{Atkinson}. K{\' e}zdy, Snevily 
and Wang~\cite{KSW} have shown that for any $r,s\ge0$, the basis of the class 
$\C{r}{s}$ is finite; however, they also pointed out that the basis of 
$\C{2}{1}$ has more than 100 permutations.

Guillemot and Vialette~\cite{GV09_321} have shown that $\C{2}{0}$-PPM is
polynomial-time solvable. A different, faster algorithm for $\C{2}{0}$-PPM has
been described by Albert et al.~\cite{ALLV_321}. By a similar approach, 
Albert et al.~\cite{ALLV_321} have also obtained a polynomial algorithm for 
$\C{1}{1}$-PPM.

It was an open problem to determine whether there is any proper 
permutation class $\cC$ for which $\cC$-Pattern PPM (or even 
$\cC$-PPM) is NP-complete~\cite{Dagstuhl,ALLV_321}. Our first main result 
solves this problem.

\begin{theorem}\label{thm-main}
 It is NP-complete to decide, for a pattern $\pi\in\C{2}{0}$ and a text 
$\tau\in\C{3}{0}$, whether $\pi$ is contained in~$\tau$.
\end{theorem}

Consequently, $\C{2}{0}$-Pattern PPM as well as 
$\C{3}{0}$-PPM are NP-complete. These results are tight in the sense that both 
$\C{1}{0}$-Pattern PPM and $\C{2}{0}$-PPM are polynomial-time solvable, as 
mentioned above. 

We obtain similar results when $\C{2}{0}$ and $\C{3}{0}$ are replaced 
with $\C{1}{1}$ and $\C{2}{1}$, respectively. In fact, here we can 
be even more restrictive. Let $\Cplus$ be the class $\C{1}{1}\cap\Av(3142)$.

\begin{theorem}\label{thm-skew}
 It is NP-complete to decide, for a pattern $\pi\in\Cplus$ and a text 
$\tau\in\C{2}{1}$, whether $\pi$ is contained in~$\tau$.
\end{theorem}

This result again implies that $\Cplus$-Pattern PPM and $\C{2}{1}$-PPM are 
NP-complete, in contrast with the polynomiality of $\C{1}{0}$-Pattern PPM,  
$\C{1}{1}$-PPM and $\C{2}{0}$-PPM.

Combining Theorems~\ref{thm-main} and~\ref{thm-skew} with previously known 
polynomial cases of $\cC$-Pattern PPM, we obtain a complexity dichotomy of 
$\cC$-Pattern PPM for principal classes~$\cC$.

\begin{theorem}\label{thm-principal}
Let $\alpha$ be a permutation. The problem $\Av(\alpha)$-Pattern PPM 
is polynomial-time solvable for $\alpha\in\{1,12,21,132,213,231,312\}$ and 
NP-complete for any other~$\alpha$.
\end{theorem}

We also obtain new tractability results, which show that the NP-hardness 
results 
for $\C{2}{0}$-Pattern PPM and $\Cplus$-Pattern PPM are tight in an even 
stronger sense than suggested above.

\begin{theorem}\label{veta_polynomialni}
If $\cC$ is a proper subclass of $\C{2}{0}$ 
then $\cC$-Pattern PPM can be solved in polynomial time.
\end{theorem}
\begin{theorem}\label{veta_polynomialni_skew}
If $\cC$ is a proper subclass of $\Cplus$
then $\cC$-Pattern PPM can be solved in polynomial time.
\end{theorem}

In Section~\ref{sec-321}, we give the proof of Theorem~\ref{thm-main}. Then, in 
Section~\ref{sec-poly}, we prove Theorem~\ref{veta_polynomialni}. In the 
subsequent section, we explain how these proofs can be adapted to 
patterns from the class $\Cplus$ (and texts from~$\C{2}{1}$), and deduce 
Theorems \ref{thm-skew}, \ref{thm-principal} and~\ref{veta_polynomialni_skew}.

\section{\texorpdfstring{Hardness of $\Av(321)$-Pattern PPM}
{Hardness of Av(321)-Pattern PPM}}\label{sec-321}

Our goal is to show that for a pattern $\pi\in \C 2 0 $ and a text 
$\tau\in\C 3 0$, it is NP-complete to decide whether $\pi$ is contained in 
$\tau$. Since the problem is clearly in NP, we focus on proving its 
NP-hardness. We take inspiration from the NP-hardness proof given by Bose, Buss 
and Lubiw~\cite{BBL} for general permutations and adapt it to the proper classes 
$\C 2 0$ and $\C 3 0$. We proceed by reduction from the classical NP-complete 
problem 
3-SAT, whose input is a 3-CNF formula $\Phi$, and the goal is to determine 
whether $\Phi$ is satisfiable.

We first introduce several auxiliary notions. Let
$\alpha=\alpha(1),\alpha(2),\dots,\alpha(n)$ be a permutation of length~$n$.
For a pair of elements $\alpha(i),\alpha(j)$, we say that $\alpha(i)$ is
\emph{above} $\alpha(j)$ (and $\alpha(j)$ is \emph{below} $\alpha(i)$), if
$\alpha(i)>\alpha(j)$. Likewise, $\alpha(i)$ is \emph{left} of $\alpha(j)$ (and
$\alpha(j)$ is \emph{right} of $\alpha(i)$) if $i<j$. For disjoint sets
$X,Y\subseteq\{\alpha(1),\alpha(2),\dots,\alpha(n)\}$ we say that $X$ is 
\emph{above} $Y$ 
if every element of $X$ is above all the elements of~$Y$, and similarly for the
other directions.

If $p=(\alpha(i),\alpha(j))$ is a pair of elements of
$\alpha$ and $\alpha(k)$ is another element, we say that 
\emph{$\alpha(k)$ is sandwiched by $p$ from below} if $i<k<j$ and 
$\alpha(k)$ is above $p$; see Figure~\ref{fig-sandwich}, left. Similarly, we say 
that
\emph{$\alpha(k)$ 
is sandwiched by $p$ from the left} if $\alpha(i)<\alpha(k)<\alpha(j)$ and
$\alpha(k)$ is to the right of~$p$; see Figure~\ref{fig-sandwich}, right. 
Analogously, we also define
sandwiching from the right or from above. More generally, a set $A$ of elements
of $\alpha$ is \emph{sandwiched from below by $p$} if each element of $A$ is 
sandwiched from below by~$p$, and similarly for the other directions.

A pair of elements $(\alpha(i),\alpha(j))$ is an \emph{increase} in $\alpha$ if
$i<j$ and $\alpha(i)<\alpha(j)$; in other words, an increase is an occurrence of
the permutation 12 in~$\alpha$.

\begin{figure}
\begin{center}
\includegraphics{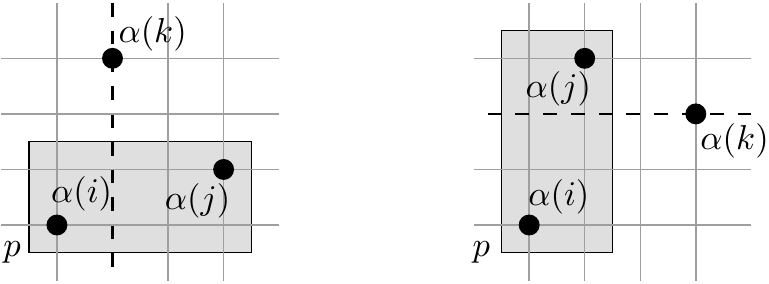}
\caption{Left: an element $\alpha(k)$ sandwiched from below by a pair of 
elements $p=(\alpha(i),\alpha(j))$. Right: $\alpha(k)$ sandwiched from the left 
by $p$.}
\label{fig-sandwich}
\end{center}
\end{figure}

A \emph{staircase of $k$ steps} is a sequence $S$ of $2k$ disjoint increases 
\[
S=(q_1,p_1,q_2,p_2,\dotsc,p_{k-1},q_k,p_k),
\] 
with the following properties (see Figure~\ref{fig-stair}): 
\begin{itemize}
\item For every $i\in [k]$, $p_i$ is sandwiched by $q_i$ 
from the left, and for $i>1$, $q_i$ is sandwiched by $p_{i-1}$ from below.
\item For every $i\in [k-1]$, $q_{i+1}$ is to the 
right and above $q_{i}$, and $p_{i+1}$ is to the right and above $p_i$. In 
particular, the elements of $q_1,q_2,\dotsc,q_k$, as well as the elements 
$p_1,p_2,\dotsc,p_k$ form an increasing sequence of length~$2k$. 
\end{itemize}
We call $q_i$ the \emph{$i$th outer bend} of $S$, and $p_i$ the \emph{$i$th 
inner bend}. Additionally, we call $q_1$ the \emph{base} of $S$ and $p_k$ the 
\emph{top} of~$S$.

\begin{figure}
\begin{center}
\includegraphics{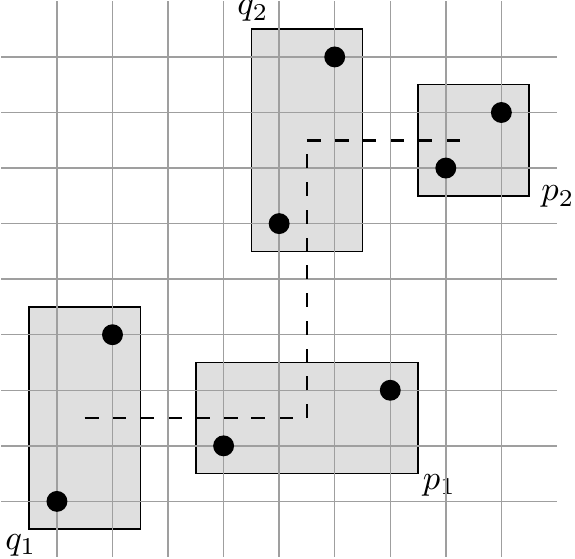}
\caption{A staircase of two steps. The shaded increases are 
the bends.}
\label{fig-stair}
\end{center}
\end{figure}

A \emph{double staircase} of $k$ steps is a pair $(S,S')$ of disjoint 
staircases of $k$ steps, with
$S=(q_1,\allowbreak p_1,\allowbreak q_2,\allowbreak p_2,\dotsc, 
p_{k-1},\allowbreak q_k,\allowbreak p_k)$ and $S'=(q'_1,\allowbreak 
p'_1,\allowbreak q'_2,\allowbreak p'_2,\dotsc, 
p'_{k-1},\allowbreak q'_k,\allowbreak p'_k)$, satisfying the following 
properties (see Figure~\ref{fig-stairs}):
\begin{itemize}
\item For each $i\ge 1$, the increase $p'_i$ is above and to the right of 
$p_i$, 
and also below (and necessarily to the right of) $q_{i+1}$. For $i<k$, $p'_i$ 
is 
also below and to the left of $p_{i+1}$. 
 \item Similarly, for $i\ge 1$, $q'_i$ is above and to the right of $q_i$, and 
also to the left (and necessarily above) $p_i$. For $i<k$, $q'_i$ is below and 
to the left of $q_{i+1}$.
\end{itemize}

An \emph{$m$-fold staircase} of $k$ steps is an $m$-tuple 
$S=(S_1,S_2,\dotsc,S_m)$ 
of staircases such that for each $i,j\in[m]$ with $i<j$, the pair $(S_i,S_j)$ 
is 
a double staircase of $k$ steps. The \emph{$i$th inner bend of $S$} is the 
union of the $i$th inner bends of the staircases $S_1,S_2,\dotsc,S_m$; the 
outer 
bends, the base and the top of $S$ are defined analogously.

\begin{figure}
\begin{center}
\includegraphics{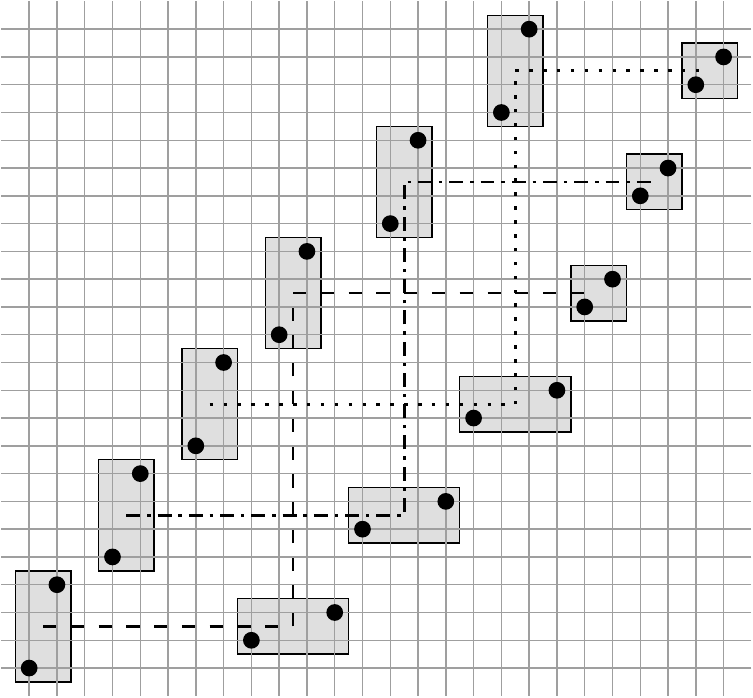}
\caption{An example of a $3$-fold staircase of two steps.}
\label{fig-stairs}
\end{center}
\end{figure}

Notice that an $m$-fold staircase avoids the pattern $321$. Moreover, the union 
of its outer bends, as well as the union of its inner bends each form an 
increasing subsequence.

\subsection{The Reduction}

We now describe the reduction from 3-SAT. Let $\Phi$ be a given 
3-CNF formula. Suppose that $\Phi$ has $v$ variables $x_1, x_2,\dotsc,x_v$ and 
$c$ clauses $K_1, K_2,\dotsc,K_c$. We will assume, without loss of generality, 
that each clause contains exactly three literals, and no variable appears 
in a single clause more than once. We will construct two permutations 
$\pi = \pi(\Phi)\in \C 2 0$ and 
$\tau = \tau(\Phi)\in\C 3 0$, such that $\Phi$ is 
satisfiable if and only if $\tau$ contains~$\pi$. 

The overall structure of $\pi$ and $\tau$ is depicted in 
Figure~\ref{fig-overview}.

\begin{figure}
\begin{center}
\includegraphics{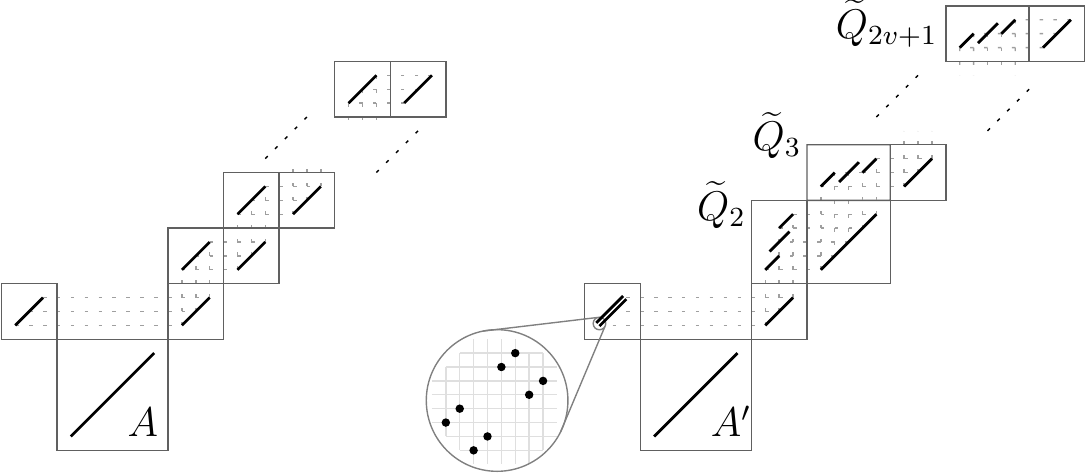}
\caption{An overview of the general structure of $\pi$ (left)
and~$\tau$ (right).}
\label{fig-overview}
\end{center}
\end{figure}

The pattern $\pi$ is the disjoint union of a $v$-fold staircase 
$X=(X_1,\allowbreak X_2,\dotsc,\allowbreak X_v)$ of $2c+1$ steps with an 
increasing sequence $A$ (called 
the \emph{anchor} of $\pi$) of length $M$, where $M$ is a sufficiently large 
value to be specified later. 
Moreover, the sequence $A$ is below $X$, to the right of the base of $X$, and 
to the left of the first inner bend of~$X$. This determines $\pi$ uniquely up 
to the value of $M$, and we may observe that $\pi$ avoids~$321$. We will say 
that the staircase $X_i$ \emph{represents} the variable $x_i$ of~$\Phi$.

We now describe the text $\tau$. As a starting point, we first build a 
permutation $\btau$ and then explain how to modify $\btau$ to obtain~$\tau$.

The permutation $\btau$ is the disjoint union of a $2v$-fold staircase 
$Y=(\xt_1,\allowbreak\xf_1,\allowbreak\xt_2,\xf_2,\dotsc,\allowbreak\xt_v,
\xf_v)$ of $2c+1$ steps, and an increasing sequence $A'$ (the \emph{anchor} of 
$\btau$) of length~$M$. The sequence $A'$ is below the $2v$-fold 
staircase, to the right of the base of $Y$, and to the left of the first outer 
bend of~$Y$. 

Each of the $2v$ staircases in $Y$ will represent one of the $2v$ possible 
literals, with $\xt_i$ representing the literal $x_i$ and $\xf_i$ representing 
$\lnot x_i$.

We now modify $\btau$ to obtain the actual text~$\tau$. The modification 
proceeds in two steps. 

In the first step, we change the relative position of the bases of the 
individual staircases in~$Y$. For every $i=1,2,\dotsc,v$, let $b_i^\text{T}$ 
and 
$b_i^\text{F}$ be the respective bases of $\xt_i$ and $\xf_i$. In $\btau$, the 
two bases together form an occurrence of $1234$. We modify their relative 
horizontal position by moving $b_i^\text{T}$ to the right of $b_i^\text{F}$, so 
that in $\tau$ the two bases will form an occurrence of $3412$. The relative 
position of $b_i^\text{T}$ and $b_i^\text{F}$ to the remaining elements remains 
unchanged; in other words, the four elements of $b_i^\text{T}\cup b_i^\text{F}$ 
occupy the same four rows and the same four columns as before.

Before we describe the second step of the construction of $\tau$, we introduce 
the following notion: suppose $S=(q_1,p_1,q_2,p_2,\dotsc,q_k,p_k)$ is a 
staircase of $k$ steps, and let $q_i$, $p_i$ and $q_{i+1}$ be three 
consecutive bends of $S$, with $i>1$. A \emph{bypass} of $q_i,p_i,q_{i+1}$ in 
$S$ is a sequence of three increases $q'_i,p'_i,q'_{i+1}$ disjoint from $S$ and 
satisfying these properties:
\begin{itemize}
\item The sequence $S'$ obtained from $S$ by replacing $q_i$ with $q'_i$, 
$p_i$ with $p'_i$ and $q_{i+1}$ with $q'_{i+1}$ is again a staircase with $k$ 
steps, and
\item $q_i$ is above and to 
the left of $q'_i$, $p_i$ is above and to the right of $p'_i$, and $q'_{i+1}$ 
is sandwiched by $q_{i+1}$ from the right (see Figure~\ref{fig-bypass}).
\end{itemize}
Note that these conditions determine the relative positions of the elements of 
$S\cup S'$ uniquely. We call the pair $(q_i,q'_i)$ the \emph{fork} of the 
bypass, and the pair $(q_{i+1},q'_{i+1})$ the \emph{merge} of the bypass.

\begin{figure}
\begin{center}
\includegraphics{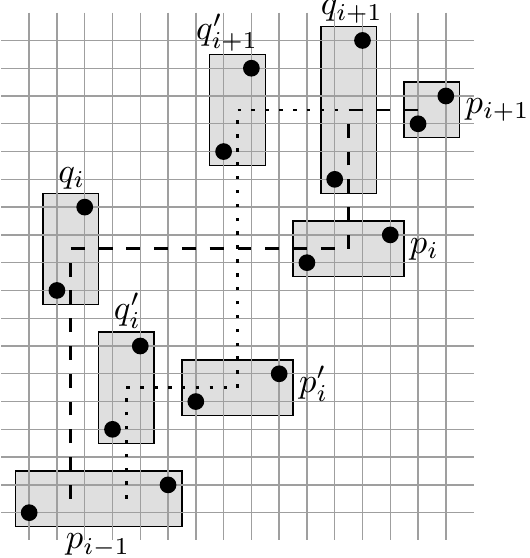}
\caption{A bypass.}
\label{fig-bypass}
\end{center}
\end{figure}

\begin{lemma}\label{lem-bypass}
Let $S=(q_1,p_1,q_2,p_2,\dotsc,q_k,p_k)$ be a staircase of $k$ steps, let 
$i\in\{2,3,\dotsc,k-1\}$ be arbitrary, and let $S'$ be obtained from $S$ by 
replacing the three bends $q_i,p_i,q_{i+1}$ with a bypass $q'_i,p'_i,q'_{i+1}$. 
Let $S''=(q''_1,\allowbreak p''_1,\allowbreak q''_2,\allowbreak 
p''_2,\dotsc,\allowbreak q''_k,\allowbreak p''_k)$ be a staircase of $k$ steps 
such 
that $p''_{i-1}=p_{i-1}$, and moreover, each of the $4k$ elements of $S''$ also 
belongs to $S\cup S'$. Then the sequence of bends 
$(q''_i,p''_i,q''_{i+1},p''_{i+1})$ is equal either to 
$(q_i,p_i,q_{i+1},p_{i+1})$ or to $(q'_i,p'_i,q'_{i+1},p'_{i+1})$.
\end{lemma}
\begin{proof}
By assumption, $p_{i-1}=p''_{i-1}$. There are four elements of $S\cup S'$ 
(namely $q_i\cup q'_i$) sandwiched from below by $p''_{i-1}$, and two of them 
must form $q''_i$. However, since $q_i$ and $q'_i$ are the only two increases 
in 
$q_i\cup q'_i$, we have either $q''_i=q_i$ or $q''_1=q'_1$. In either case, 
there is a unique increase sandwiched by $q''_i$ from the left, and this 
increase must therefore be $p''_i$, and by the same argument, both $q''_{i+1}$ 
and $p''_{i+1}$ are determined uniquely. The lemma follows.
\end{proof}

We continue by the second step of the construction of the text~$\tau$. The 
general approach is to modify certain parts of the staircases by adding 
bypasses whose structure depends on the clauses of the formula~$\Phi$. Recall 
that $\Phi$ has $c$ clauses $K_1, K_2,\dotsc,K_c$, and that the $2v$-fold 
staircase $Y$ has $2c+1$ steps.

Let $Q_m$ and $P_m$ denote the $m$th outer bend and the $m$th inner bend 
of~$Y$, 
respectively. For every $t\in[c]$, we will associate to the clause $K_t$ the 
sequence of three bends $Q_{2t}, P_{2t}, Q_{2t+1}$ of~$Y$. 

Let $K_t$ be a clause of $\Phi$ of the form $(L_i\lor L_j\lor L_k)$, with  
$L_i\in\{x_i,\lnot x_i\}$, $L_j\in\{x_j,\lnot x_j\}$, and $L_k\in\{x_k,\lnot 
x_k\}$, for some $i<j<k$. 

Let $\xp_i\in\{\xt_i,\xf_i\}$ denote the staircase representing the literal
$L_i$, and let $\xm_i\in\{\xt_i,\xf_i\}\setminus\{\xp_i\}$ be the staircase
representing the other literal containing~$x_i$. We define $\xp_j$, $\xm_j$,
$\xp_k$ and $\xm_k$ analogously. 

We will now add bypasses to (some of) the staircases $\xt_1,\allowbreak 
\xf_1,\allowbreak\xt_2,\allowbreak \xf_2, \dots,\allowbreak \xt_v,\allowbreak 
\xf_v$ into the three bends
associated to $K_t$, by inserting new bends into
$Q_{2t}$, $P_{2t}$ and $Q_{2t+1}$, and sometimes changing the
relative position of existing bends. The choice of the relative positions of
these bends is the key aspect of our reduction. 

\begin{figure}
\begin{center}
\includegraphics{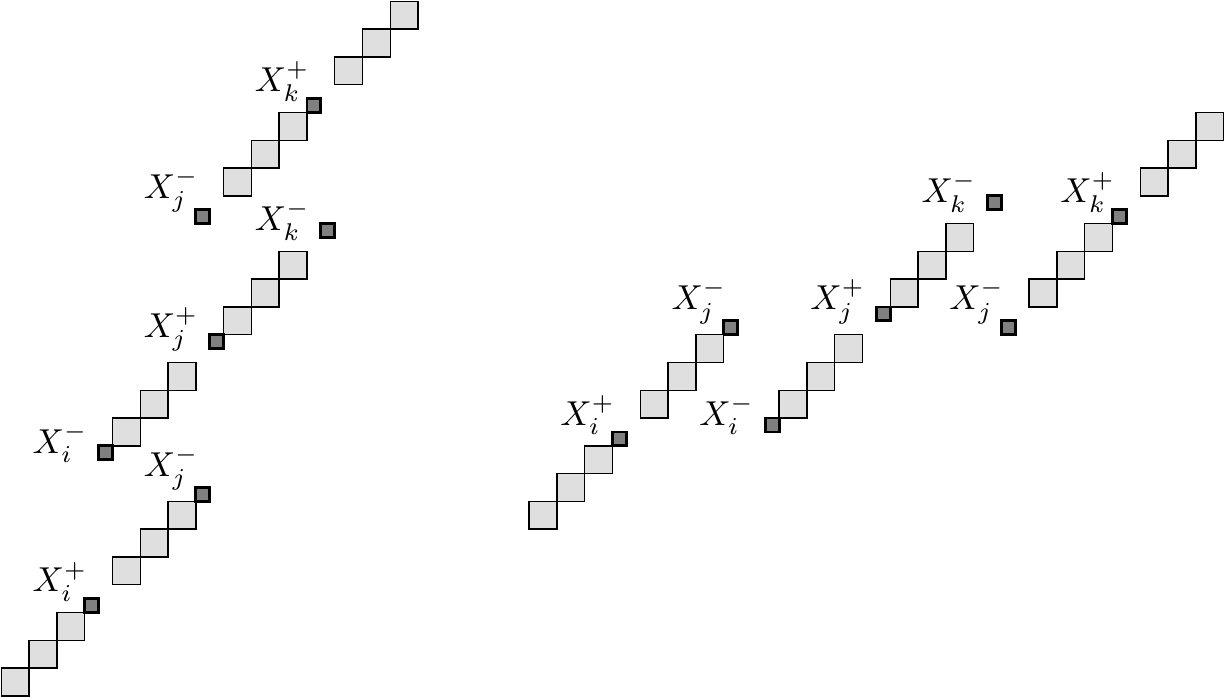}
\caption{The fork gadget (left) and the corresponding merge gadget (right) for 
the clause $(x_i \vee \neg x_j \vee x_k)$. Each light square indicates the 
positions of two bends on the two staircases representing a variable other than 
$x_i$, $x_j$ and $x_k$. Each dark square indicates the position of one bend on a 
staircase representing one of the literals $x_i,\lnot x_i, x_j,\lnot 
x_j,x_k,\lnot x_k$. 
}
\label{fig-gadget}
\end{center}
\end{figure}

We first describe how to modify $Q_{2t}$. We will replace 
$Q_{2t}$ by a so called \emph{fork gadget} $\moQ_{2t}$ containing the 
$(2t)$th outer bends of all the staircases $\xf_i$ and $\xt_i$ and also the 
corresponding bends of their bypasses, if any. The gadget $\moQ_{2t}$ is a 
union 
of three disjoint increasing sequences, which we call the \emph{top level}, the 
\emph{middle level} and the \emph{bottom level}; see Figure~\ref{fig-gadget}. 
As 
the names suggest, the top level is above the middle one, and the middle one is 
above the bottom one.

The top level contains, in left-to-right order, the bend of $\xm_j$, the bends 
of $\xt_{j+1},\allowbreak \xf_{j+1},\allowbreak \xt_{j+2},\allowbreak 
\xf_{j+2},\dotsc,\xt_{k-1},\xf_{k-1}$, followed by the bend of
$\xp_k$, and finally the bends of 
$\xt_{k+1},\allowbreak\xf_{k+1},\allowbreak\xt_{k+2},\allowbreak\xf_{k+2},\dotsc
,\allowbreak\xt_v,\xf_v$.

The middle level contains, left to right, the bends of 
$\xm_i,\allowbreak\xt_{i+1},\allowbreak\xf_{i+1}, \dotsc,\allowbreak 
\xt_{j-1},\allowbreak \xf_{j-1},\allowbreak \xp_j$, the 
bypasses of $\xt_{j+1},\allowbreak \xf_{j+1},\allowbreak\xt_{j+2},\allowbreak 
\xf_{j+2},\dotsc,\allowbreak\xt_{k-1},\allowbreak\xf_{k-1}$, and 
finally the bend of~$\xm_k$.

The bottom level contains the bends of $\xt_1, \xf_1, \xt_2, \xf_2, \dotsc, 
\xt_{i-1}, 
\xf_{i-1},\allowbreak \xp_i$, the bypasses of $\xt_{i+1},\allowbreak 
\xf_{i+1},\allowbreak \xt_{i+2},\allowbreak \xf_{i+2},\dotsc,\allowbreak 
\xt_{j-1},\allowbreak \xf_{j-1}$, and the bypass of~$\xm_j$.

Note that $\moQ_{2t}$ does not contain the permutation $321$: indeed, 
any occurrence of $321$ would have to intersect all three levels, since each 
level is an increasing sequence. However, the only elements of the top level 
that have any element of the bottom level to the right of them are the two 
elements of the bend of~$\xm_j$, and neither of these two elements belongs to 
an 
occurrence of $321$.

We then replace the inner bend $P_{2t}$ by a single increasing sequence 
$\moP_{2t}$ containing the first inner bend of all the staircases and bypasses 
emerging from $\moQ_{2t}$. Note that the vertical position of these staircases 
is already determined by~$\moQ_{2t}$. 

Finally, we replace $Q_{2t+1}$ by a \emph{merge gadget} $\moQ_{2t+1}$, in which 
all the bypasses that forked in $\moQ_{2t}$ will be merged. Note that the 
horizontal relative position of the bends in $\moQ_{2t+1}$ is determined 
uniquely by $\moP_{2t}$, while their relative vertical position is fixed 
by~$P_{2t+1}$. Essentially, $\moQ_{2t+1}$ is similar to a transpose of 
$\moQ_{2t}$ along the North-East diagonal, except that a fork of a bypass forms 
an occurrence of $3412$ in $\moQ_{2t}$, whereas the corresponding merge forms 
an 
occurrence of $2314$ in~$\moQ_{2t+1}$. In particular, we may again see that 
$\moQ_{2t+1}$ is $321$-free.

Performing the above-described modifications for each $t\in\{1,2,\dotsc,c\}$, 
we obtain the text~$\tau$. We let $\mo{Y}$ denote the part of $\tau$ that was 
obtained by replacing the bends of $Y$ by the corresponding gadgets; in other 
words, $\mo{Y}$ contains all the elements of $\tau$ except the anchor.

We let $\moxt_i$ (or $\moxf_i$) denote the subpermutation obtained as the union 
of $\xt_i$ (or $\xf_i$) and all the bypasses added to it.

It remains to determine the value of $M$, i.e., the length of the anchors of 
$\pi$ and~$\tau$. We choose $M$ to be equal to $|\mo{Y}|+1$, that is, the 
elements of the anchor $A'$ will outnumber the remainder of~$\tau$.

\subsection{Correctness}
We now verify that the construction has the desired properties.

\begin{proposition}\label{pro-correct}
Let the formula $\Phi$, pattern $\pi$ and text $\tau$ be as described in the 
previous subsection. 
Then $\Phi$ is satisfiable if and only if $\tau$ contains~$\pi$.
\end{proposition}
\begin{proof}
Suppose that $\Phi$ is satisfiable. Fix any satisfying assignment, and 
represent 
it by a function $\phi\colon\{1,2,\dotsc,v\}\to\{\text{T},\text{F}\}$, where 
$\phi(i)=\text{T}$ if and only if the variable $x_i$ is true in the chosen 
satisfying assignment. We obtain a copy of $\pi$ in $\tau$ as follows. The 
elements of the anchor $A$ of $\pi$ will be mapped to the anchor $A'$ 
of~$\tau$. 

It remains to map the $v$-fold staircase $X=(X_1,X_2,\dotsc,X_v)$ to $\mo{Y}$. 
We 
will 
show that it is possible to map each $X_i$ into $\mo{X}^{\phi(i)}_i$. To obtain 
such a mapping, we need to decide, for every bypass appearing in 
$\mo{X}^{\phi(i)}_i$, whether to map $X_i$ to the bends of the bypass or to the 
bends of $X^{\phi(i)}_i$. The decision can be made for each bypass 
independently, but we need to make sure that for each gadget $\moQ_m$ in 
$\mo{Y}$, the bends mapped into $\moQ_m$ will form an increasing sequence. 

It can be verified by a routine case analysis that such a choice is always 
possible. To see this, suppose that $K_t$ is a clause $(L_i\lor L_j\lor 
L_k)$ whose literals contain variables $x_i$, $x_j$ and $x_k$, respectively, 
with $i<j<k$. Assume for instance, that the assignment $\phi$ satisfies $L_i$ 
but not $L_j$ and $L_k$. Then the $(2t)$th outer bends of $X_i$, $X_j$ and 
$X_k$ must be mapped to the bends of $\moxp_i$,  $\moxm_j$ and  $\moxm_k$ in 
$\moQ_{2t}$. The bends of $\moxp_i$ and $\moxm_k$ are unique, and to 
preserve monotonicity, we need to choose the bend of $\moxm_j$ in the bottom 
level of $\moQ_{2t}$, i.e., the bypass of $\xm_j$. Then all the staircases 
$X_1,X_2,\dotsc,X_j$ may be mapped to bends in the bottom level of $\moQ_{2t}$, 
while $X_{j+1},X_{j+2},\dotsc,X_v$ may be mapped to bends in the middle level, 
preserving monotonicity. 

Notice that the position of the bends in $\moQ_{2t+1}$ is the transpose along 
the North-East diagonal of their position in $\moQ_{2t}$, and in particular, 
the 
bends will form an increasing sequence in $\moQ_{2t+1}$ as well. 

We conclude that if $\phi$ is a satisfying assignment of $\Phi$, then $\pi$ 
occurs in~$\tau$. For the converse, suppose that $\pi$ has an occurrence 
in~$\tau$. Since the anchor $A$ of $\pi$ is longer than $\mo{Y}$, at least one 
element of $A$ must be mapped to an element of~$A'$. In particular, all the 
elements to the left and above $A$ are mapped to elements to the left and 
above~$A'$. In other words, the base of $X$ gets mapped to a subset of the 
base of~$\mo{Y}$. 

Recall that the base of $\mo{Y}$ is an increasing sequence of $v$ blocks of size 
$4$, where the $i$th block is the union of the base $b_i^\text{T}$ of $\moxt_i$ 
with the base $b_i^\text{F}$ of $\moxf_i$. Recall also, that each of these $v$ 
blocks is order isomorphic to $3412$, and in particular, the longest increasing 
subsequence of each block has size $2$, and there are exactly two such 
subsequences, namely $b_i^\text{T}$ and~$b_i^\text{F}$. 

This implies that any increasing subsequence of length $2v$ of the base of 
$\mo{Y}$ contains exactly one of $b_i^\text{T}$ and~$b_i^\text{F}$ for 
each~$i$. In particular, in an occurrence of $\pi$ in $\tau$, the base of 
$X_i$ is mapped either to $b_i^\text{T}$ or to $b_i^\text{F}$. Fix an 
occurrence of $\pi$ in $\tau$ and use it to define a truth assignment 
$\phi\colon[v]\to\{\text{T},\text{F}\}$, so that the base of $X_i$ is mapped 
to~$b_i^{\phi(i)}$. 

We claim that the assignment $\phi$ satisfies~$\Phi$. To see this, we first 
argue that for every $i\in[v]$, the elements of $X_i$ are mapped to elements of 
$\mo{X}^{\phi(i)}_i$, and more specifically, each (inner or outer) 
bend on $X_i$ is mapped either to the corresponding bend of 
$X^{\phi(i)}_i$ or to the corresponding bend of a bypass of~$X^{\phi(i)}_i$. We 
have already seen that this is the case for the base of~$X_i$. To show that it 
holds for the remaining bends also, we may proceed by induction and simply note 
that the only elements sandwiched from below by an inner bend in 
$\mo{X}^{\phi(i)}_i$ are the elements of the subsequent outer bend of 
$\mo{X}^{\phi(i)}_i$, or perhaps a pair of outer bends forming the fork of a 
bypass. An analogous property holds for outer bends as well. Using 
Lemma~\ref{lem-bypass}, we may then conclude that the bends of $X_i$ map to 
corresponding bends in~$\mo{X}^{\phi(i)}_i$. 

To see that $\phi$ is satisfying, assume for contradiction that there is a 
clause $K_t$ whose literals involve the variables $x_i, x_j$ and $x_k$, and 
neither of these literals is satisfied by~$\phi$. It follows that inside the 
gadget 
$\moQ_{2t}$, the bends of $X_i$, $X_j$ and $X_k$ must map to the bends of 
$\moxm_i$, $\moxm_j$ and $\moxm_k$. However, no three such bends form an 
increasing sequence in $\moQ_{2t}$, whereas the corresponding bends form an 
increasing sequence in~$\pi$. This contradiction completes the proof of the 
proposition.
\end{proof}

\begin{proof}[Proof of Theorem~\ref{thm-main}]
Clearly, the problem is in the class NP. It is easy to observe that in our 
reduction, $\pi$ belongs to $\C 2 0$. To see that $\tau$ belongs to 
$\C 3 0$, it suffices to note that the gadgets $\moQ_m$ used in the 
construction of $\tau$ all avoid $321$, and that the base of $\mo{Y}$ avoids 
$321$ as well. Note also that the base of $\mo{Y}$ is to the left and below all 
the gadgets~$\moQ_m$. Clearly, $\pi$ and $\tau$ can be constructed from $\Phi$ 
in polynomial time, and the correctness of the reduction follows from 
Proposition~\ref{pro-correct}.
\end{proof}

\section{\texorpdfstring{Patterns from a proper subclass of $\Av(321)$}
{Patterns from a proper subclass of Av(321)}}\label{sec-poly}

In this section we prove Theorem~\ref{veta_polynomialni}.
We rely on a result by Ahal and Rabinovich~\cite{AR08_subpattern}, who showed 
that for patterns
with bounded ``treewidth'', the PPM problem can be solved in polynomial time. 
Our main contribution is in showing that patterns in $\Av(321)$ of large ``treewidth'' contain a large universal pattern, containing all patterns in $\Av(321)$ of a given size. To show this, we use the grid minor theorem by Robertson and Seymour~\cite{Ch16_grid,RS86_minorsV_grid}.

\subsection{Permutations and treewidth}

The following definition was introduced by Ahal and Rabinovich~\cite[Definition 2.3]{AR08_subpattern}.
For a $k$-permutation $\pi$, a graph $G_{\pi}$ is defined as follows; see 
Figure~\ref{obr_G_pi}. The vertices of $G_{\pi}$ are the numbers $1,2,\dots,k$, 
interpreted as the elements of $\pi=\pi(1),\pi(2),\dots,\pi(k)$. Two vertices 
$\pi(i),\pi(j)$ are connected by an edge if $|i-j|=1$ or 
$|\pi(i)-\pi(j)|=1$. We say that an edge between $\pi(i)$ and $\pi(j)$ is
\emph{red} if $|\pi(i)-\pi(j)|=1$, and it is \emph{blue} if
$|i-j|=1$. Note that an edge can be both red and blue. 
Clearly, the edges of each color form a Hamiltonian path in~$G_{\pi}$. We note 
that in our definition, $G_{\pi}$ is a graph, whereas Ahal and 
Rabinovich~\cite{AR08_subpattern} defined $G_{\pi}$ as a multigraph. Also, we 
label the vertices of $G_{\pi}$ by their value rather than position in~$\pi$.

\begin{figure}
\begin{center}
\includegraphics{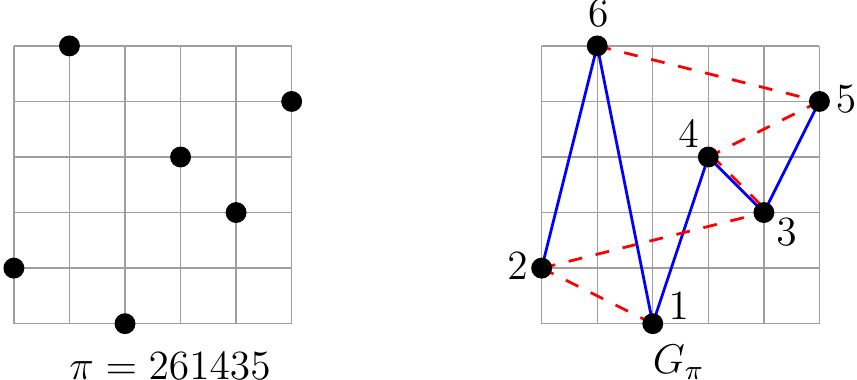}
\end{center}
\caption{A permutation $\pi$ and the graph $G_{\pi}$.}
\label{obr_G_pi}
\end{figure}

Let $\mathrm{tw}(G)$ be the treewidth of a graph $G$. The main result of Ahal and Rabinovich~\cite{AR08_subpattern} can be stated in the following form.

\begin{theorem}[{\cite[Theorem 2.10, Proposition 3.6]{AR08_subpattern}}]\label{veta_o_omezene_sirce}
For $\pi\in S_k$ and $\tau\in S_n$, the problem whether $\pi$ is contained in $\tau$ can be solved in time $O(kn^{2\cdot\mathrm{tw}(G_{\pi})+2})$.
\end{theorem}

\subsection{Treewidth, grids and walls}

The \emph{$r\times r$ grid} is the graph with vertex set $[r]\times [r]$ where 
vertices $(i,j)$ and $(i',j')$ are joined by an edge if and only if 
$|i-i'|+|j-j'|=1$. See Figure~\ref{obr_grid_wall}, left.

Robertson and Seymour~\cite{RS86_minorsV_grid} proved that for every $r$, every graph of sufficiently large treewidth contains the $r\times r$ grid as a minor. Recently, Chekuri and Chuzhoy~\cite{ChCh14_grid} showed that a treewidth polynomial in $r$ is sufficient. The upper bound has been further improved by Chuzhoy~\cite{Ch16_grid}.

\begin{theorem}[{\cite{Ch16_grid}}]\label{veta_grid}
There is a function $f\colon\mathbb{N}\rightarrow \mathbb{N}$ satisfying $f(r) 
\le r^{19}(\log{r})^{O(1)}$ such that every graph of treewidth at least $f(r)$ 
contains the $r\times r$ grid as a minor.
\end{theorem}

\begin{figure}
\begin{center}
\includegraphics{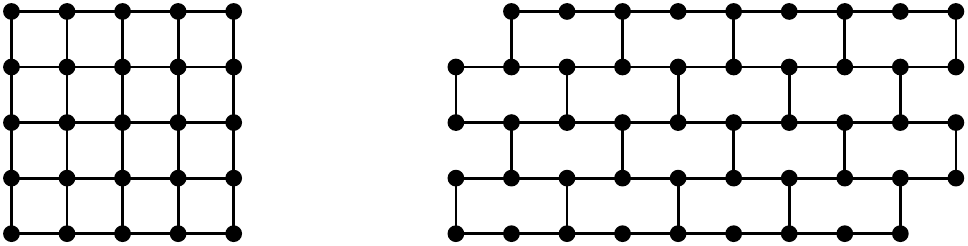}
\end{center}
\caption{Left: a $5\times 5$ grid. Right: an elementary wall of height $4$.}
\label{obr_grid_wall}
\end{figure}

Since grids have vertices of degree $4$, it is more convenient to consider their subgraphs of maximum degree $3$, called walls.
Let $r\ge 2$ be even. An \emph{elementary wall of height 
$r$}~\cite{RS95_minorsXIII_wall} is obtained from the $(r+1)\times(2r+2)$ grid 
by removing two opposite corners $(1,2r+2)$ and $(r+1,1)$, edges 
$\{(i,2j),(i+1,2j)\}$ for $i$ odd and $1\le j\le r+1$, and edges 
$\{(i,2j-1),(i+1,2j-1)\}$ for $i$ even and $1\le j\le r+1$. That is, an 
elementary wall of height $r$ is a planar graph of maximum degree $3$, which can 
be drawn as a ``wall'' with $r$ rows of $r$ ``bricks'', where each ``brick'' is 
a face of size $6$. See Figure~\ref{obr_grid_wall}, right.
A subdivision of an elementary wall of height $r$ is called a \emph{wall of height $r$} or simply an \emph{$r$-wall}. It is well known that if $H$ is a graph of maximum degree $3$, then a graph $G$ contains $H$ as a minor if and only if $G$ contains a subdivision of $H$ as a subgraph. Therefore, a graph containing the $(2r+2)\times (2r+2)$ grid as a minor also contains an $r$-wall as a subgraph.

\subsection{\texorpdfstring{Universal patterns in $\Av(321)$}
{Universal patterns in Av(321)}}
Let $\omega$ be a finite sequence of $n$ distinct positive integers. The \emph{reduction} of $\omega$ is an $n$-permutation obtained from $\omega$ by replacing the $i$th smallest element by $i$, for every $i\in [n]$.

For every $k\in\mathbb{N}$, we define the \emph{$k$-track} as the
$k^2$-permutation that is the reduction of the sequence $(1,2k,3,2k+2,5,2k+4,\dots,k^2-1,k^2+2k-2)$ if $k$ is even, and the reduction of $(1,2k,3,2k+2,5,2k+4,\dots,k^2-k-1,k^2+k-2,k^2-k+1,k^2-k+3,\dots,k^2+k-1)$ if $k$ is odd. See Figure~\ref{obr_k_track}. The $k$-track clearly avoids $321$ since it is a union of two increasing sequences. In Lemma~\ref{lemma_ktrack} we will show that the $k$-track is a universal pattern for all 321-avoiding $k$-permutations.

\begin{figure}
\begin{center}
\includegraphics{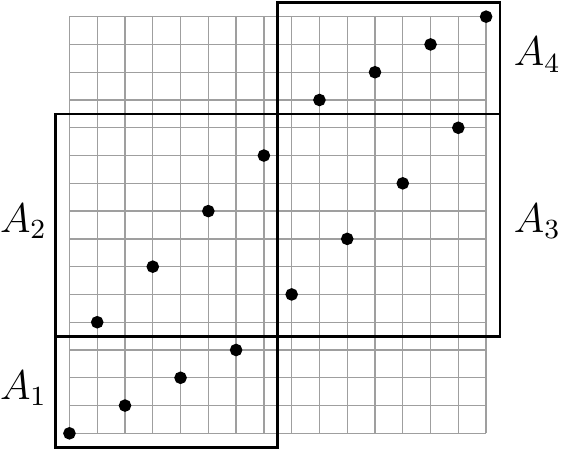}
\end{center}
\caption{The $4$-track and its stair-decomposition.}
\label{obr_k_track}
\end{figure}

We use the stair-decomposition of $321$-avoiding permutations introduced by Guillemot and Vialette~\cite{GV09_321} and independently, in a slightly different way, by Albert et al.~\cite{AABRSW10_staircase_decomp}.
Let $\pi\in\Av(321)$ be a $k$-permutation. A \emph{stair-decomposition} of $\pi$ is a partition of 
$[k]$, regarded as the set of elements of $\pi$, into possibly empty subsets 
$B_1, B_2,\dots,B_m$, for some $m$, 
such that 
\begin{itemize}
\item each $B_i$ forms an increasing subsequence in $\pi$,
\item $B_{2i}$ is above $B_{2i-1}$ for each $i \le \lfloor m/2 \rfloor$,
\item $B_{2i+1}$ is to the right of $B_{2i}$ for each $i \le \lceil m/2 \rceil-1$, and 
\item $B_{i+2}$ is above and to the right of $B_i$ for each $i\le m-2$.
\end{itemize}

The subsets
$B_i$ are called the \emph{blocks} of the stair-decomposition. Sometimes it will 
be convenient to refer to blocks $B_0$ or $B_{m+1}$, which we define as empty 
sets.

The $k$-track has a stair-decomposition into $k$ blocks 
$A_1,A_2,\dots,A_{k}$, each containing exactly $k$ elements; see Figure~\ref{obr_k_track}. Moreover, for every $i\le k/2$, the subset 
$A_{2i-1}\cup A_{2i}$ forms a \emph{vertical alternation} in the $k$-track, that is, 
$A_{2i}$ is above $A_{2i-1}$ and the elements from $A_{2i-1}$ alternate from left to right with the elements from $A_{2i}$ in the $k$-track. Similarly, for every $i\le (k-1)/2$, the subset $A_{2i}\cup A_{2i+1}$ forms a \emph{horizontal alternation} in the $k$-track, that is, $A_{2i+1}$ is to the right of $A_{2i}$ and the elements of $A_{2i}$ alternate from bottom to top with the elements of $A_{2i+1}$ in the $k$-track.

Guillemot and Vialette~\cite{GV09_321} state the following lemma without proof. 
Albert et al.~\cite{AABRSW10_staircase_decomp} give an algorithm to compute the 
stair-decomposition but do not include verification of correctness.

\begin{lemma}\label{lemma_ex_stair_dec}
Every $321$-avoiding $k$-permutation has a stair-decomposition with at most $k$ 
blocks.
\end{lemma}

\begin{proof}
Let $\pi\in\Av(321)$ be a $k$-permutation. We define a stair-decomposition of $\pi$ by a greedy algorithm. Let $B_1$ be the longest 
interval $[i]$ whose elements form an increasing subsequence in $\pi$. 
Let $\rho_1$ be the subsequence of $\rho_0=\pi$ formed by the elements in $[k] 
\setminus B_1$. Now let $B_2$ be the subset of $[k] \setminus B_1$ whose 
elements form the maximal increasing prefix of $\rho_1$. Let $\rho_2$ be the 
subsequence of $\rho_1$ obtained by removing the elements of $B_2$. We continue 
analogously. Suppose that $\rho_{2i}$ and $B_1,B_2,\dots, B_{2i}$ have been 
defined. Then let $B_{2i+1}$ be the maximal down-set of $[k]\setminus 
\bigcup_{j=1}^{2i} B_j$ forming an increasing subsequence in $\rho_{2i}$, and 
let $\rho_{2i+1}$ be the subsequence of $\rho_{2i}$ obtained by removing the 
elements of $B_{2i+1}$. Finally, let $B_{2i+2}$ be the subset of $[k] \setminus 
\bigcup_{j=1}^{2i+1} B_j$ whose elements form the maximal increasing prefix of 
$\rho_{2i+1}$, and let $\rho_{2i+2}$ be the subsequence of $\rho_{2i+1}$ 
obtained by removing the elements of $B_{2i+2}$. We continue until $\rho_{2i}$ 
or $\rho_{2i+1}$ is empty and we denote by $m$ the largest index such that $B_m$ 
is nonempty. Clearly $m\le k$.

Now we verify that $B_1,B_2,\dots,B_m$ is indeed a stair-decomposition of $\pi$. The facts that $B_{2i}$ and $B_{2i+1}$ are above $B_{2i-1}$, and that $B_{2i+2}$ and $B_{2i+1}$ are to the right of $B_{2i}$ for every $i$ follow directly from the construction. 

For every $i$, the block $B_{2i+1}$ is to the right of $B_{2i-1}$ since the set of elements above and to the left of $\max(B_{2i-1})$ forms an increasing subsequence in $\rho_{2i-2}$; a decreasing pair would form a forbidden pattern 321 with $\max(B_{2i-1})$. Finally, for every $i$, the block $B_{2i+2}$ is above $B_{2i}$ since the set of elements below and to the right of $\max(B_{2i})$ forms an increasing subsequence in $\rho_{2i-1}$; a decreasing pair would induce a forbidden pattern 321 with $\max(B_{2i})$.
\end{proof}

Albert et al.~\cite[Proposition 6]{AABRSW10_staircase_decomp} proved that each
321-avoiding permuta\-tion of size $k$ is contained in an $m$-track for some 
$m\le 2^k$. Using similar ideas, we observe the following stronger fact.

\begin{lemma}\label{lemma_ktrack}
Let $\pi\in\Av(321)$ be a $k$-permutation, and let $B_1,B_2,\dots, B_m$ be its 
stair-decomposition. Let $q=\max\{k,m\}$. Let $\tau_q$ be the $q$-track, and 
let $A_1,A_2,\dots,A_q$ be its stair-decomposition into blocks of size~$q$.
Then $\pi$ has an occurrence in $\tau_q$ in which the 
elements of $B_i$ are mapped into $A_i$ for every $i=1,2,\dots,q$. 
\end{lemma}

\begin{proof}
For $m=1$ the claim is trivial, so we assume that $m\ge 2$.

For $i \le \lfloor m/2 \rfloor$, let $<_{2i-1}$ be the left-to-right linear order of the elements from $B_{2i-1}\cup B_{2i}$ in $\pi$. Similarly, for $i \le \lceil m/2 \rceil-1$, let $<_{2i}$ be the bottom-to-top linear order of the elements from $B_{2i}\cup B_{2i+1}$ in $\pi$ (that is, $<_{2i}$ is the restriction of the standard linear order $<$ of the integers). We claim that there is a linear order $<_{\pi}$ on $[k]$ that simultaneously extends all the orders $<_1, <_2, \dots, <_{m-1}$. 
We show this by induction on $m$. For $m=2$ we can take $<_{\pi}$ as $<_1$. Now 
let $m\ge 3$ and assume that there is a linear order $<'$ extending the union 
$<_1 \cup <_2 \cup \dots \cup <_{m-2}$. This implies that there is an injective 
function $f\colon[k]\setminus B_m \rightarrow \mathbb{R}$ satisfying $a<'b 
\Leftrightarrow f(a)<f(b)$. Since $<_{m-2}$ and $<_{m-1}$ intersect at a linear 
order on $B_{m-1}$, which is a restriction of both $<$ and $<'$, we have $a 
<_{m-1} b \Leftrightarrow a<'b \Leftrightarrow f(a)<f(b)$ for every $a,b\in 
B_{m-1}$. Clearly, we can extend the function $f$ to the elements of $B_m$ so 
that $a <_{m-1} b \Leftrightarrow f(a)<f(b)$ for every $a,b\in B_{m-1}\cup B_m$. 
The order $<_{\pi}$ can now be defined as $a <_{\pi} b \Leftrightarrow 
f(a)<f(b)$ for all $a,b\in[k]$.

With the order $<_{\pi}$ at hand, we define an embedding of $\pi$ to the 
$q$-track as follows. Let $b_1,b_2, \dots, b_k$ be the permutation of $[k]$ 
satisfying $b_1 <_{\pi} b_2 <_{\pi} \dots <_{\pi} b_k$. Similarly, for every 
$i\in [q]$, let $a_{i,1},a_{i,2},\dots,a_{i,k}$ be the permutation of $A_i$ 
satisfying $a_{i,1} < a_{i,2} < \dots < a_{i,k}$. For every $j\in [k]$, let 
$i(j)$ be the index such that $b_j\in B_{i(j)}$. Then the map sending $b_j$ to 
$a_{i(j),j}$, for each $j\in [k]$, is an embedding of $\pi$ to the $q$-track.
\end{proof}

\subsection{\texorpdfstring{Patterns in $\Av(321)$ of large ``treewidth''}
{Patterns in Av(321) of large "treewidth"}}

Using the results from previous subsections, Theorem~\ref{veta_polynomialni} will follow from the following theorem.

\begin{theorem}\label{theorem_hledani_ktracku}
There is a function $g\colon\mathbb{N}\rightarrow \mathbb{N}$ satisfying $g(k) 
\le O(k^{3/2})$ with the following property. If $\pi\in\Av(321)$ is a 
permutation such that $G_{\pi}$ contains a $g(k)$-wall, then $\pi$ contains the 
$k$-track.
\end{theorem}

We will need the following lemma providing an upper bound on the treewidth of $G_{\rho}$ for $321$-avoiding permutations $\rho$.

\begin{lemma}\label{lemma_tw_blocks}
If $\rho\in\Av(321)$ has a stair-decomposition with $m$ blocks, then the 
treewidth of $G_{\rho}$ is at most $2m$. \end{lemma}

We show that, in fact, the pathwidth of $G_{\rho}$ is at most $2m$. 
Kinnersley~\cite{Kin92_vertex_sep} proved that the pathwidth of a graph $G$ is 
equal to the \emph{vertex separation number} of $G$, $\mathrm{vs(G)}$, defined 
as follows. If $G$ is a graph $G$ with $n$ vertices, $L=(v_1,v_2, \dots, v_n)$ 
is a linear ordering of the vertices and $i\in [n]$, let $V(G,L,i)$ be the set 
of vertices $v_j$ such that $j<i$ and $G$ has an edge $v_jv_k$ with $k\ge i$. 
Then let 
\[
\mathrm{vs}(G)=\min_{L}\max_i |V(G,L,i)|.
\]

\begin{proof}[Proof of Lemma~\ref{lemma_tw_blocks}]
Let $\rho\in\Av(321)$ be an $n$-permutation that has a stair-decomposition with 
blocks $B_1, B_2, \dots, B_m$. Let $<_{\rho}$ be the linear ordering of the 
elements of $\rho$ defined in the proof of Lemma~\ref{lemma_ktrack}. We claim 
that for each $i\in [n]$ we have $|V(G_{\rho},<_{\rho},i)|\le 2m$.

The graph $G_{\rho}$ is a union of two Hamiltonian paths, the \emph{blue} path 
$(\pi(1),\allowbreak \pi(2),\dots,\pi(n))$ formed by the blue edges and the 
\emph{red} path $(1,2,\dots,n)$ formed by the red edges. The blue path can be 
decomposed into $\lceil m/2 \rceil$ $<_{\rho}$-increasing paths induced by the 
subsets $B_{2j-1} \cup B_{2j}$ and $\lceil m/2 \rceil-1$ remaining blue edges. 
Similarly, the red path can be decomposed into $1+\lfloor m/2 \rfloor$ 
$<_{\rho}$-increasing paths induced by the subsets $B_{2j} \cup B_{2j+1}$ and 
$\lfloor m/2 \rfloor$ remaining red edges. All together, $G_{\rho}$ has a 
decomposition into $2m$ $<_{\rho}$-monotone paths, and each such path 
contributes at most one vertex to each $V(G_{\rho},<_{\rho},i)$.
\end{proof}

Let $\pi\in\Av(321)$ be a permutation with a stair-decomposition 
$B_1,\allowbreak B_2,\dots,\allowbreak B_m$. Let $xy$ be an edge of $G_{\pi}$ such that $x\in B_i$, 
$y\in B_j$ and $i\le j$. We say that $xy$ is \emph{good} if at least one of the 
following conditions is satisfied:
\begin{itemize}
\item $i=j$,
\item $xy$ is blue, $i$ is odd and $j=i+1$, or
\item $xy$ is red, $i$ is even and $j=i+1$.
\end{itemize}
Otherwise $xy$ is \emph{bad}.

We introduce the following notation: if $B_1,B_2,\dotsc,B_m$ is a sequence of 
sets and $i\in[m]$ an integer, then $B_{<i}$ denotes the set 
$\bigcup_{j=1}^{i-1} 
B_j$. The sets $B_{\le i}$, $B_{>i}$ and $B_{\ge i}$ are defined analogously.

\begin{lemma}\label{lemma_bad_edges}
Let $\pi\in\Av(321)$ be a permutation with a stair-decomposition $B_1,B_2,\dots,B_m$. Then 
\begin{enumerate}
\item[$1)$] $G_{\pi}$ has at most $m-1$ bad edges,
\item[$2)$] for each $i\in [m]$ there is at most one vertex $x\in B_i$ incident
to a bad edge $xy$ with $y\in B_{>i}$, and  
\item[$3)$] for each $i\in [m]$ there is at most one vertex $x\in B_i$ incident 
to a bad edge $yx$ with $y\in B_{<i}$.
\end{enumerate}
\end{lemma}

\begin{proof}
If $xy$ is a bad edge with $x\in B_i$ and $y\in B_{>i}$, then $x$ must be the 
topmost (or equivalently rightmost) vertex of~$B_i$. This shows part 2), and 
part 3) is analogous.

For part 1), we observe that every bad blue edge starts at the rightmost vertex 
of some pair of blocks $B_{2i-1}, B_{2i}$, and for each such pair with $2i<m$ 
there is at most one such edge. Similarly, every bad red edge starts at the 
topmost vertex of some pair of blocks $B_{2i}, B_{2i+1}$, and for each such 
pair with $2i+1<m$ there is at most one such edge.
\end{proof}

The following lemma is the first step towards the proof of Theorem~\ref{theorem_hledani_ktracku}. The goal is to find ``many'' vertex-disjoint paths in $G_{\pi}$ between two ``distant'' blocks of the stair-decomposition.

\begin{lemma}\label{lem-graf} 
Let $k$ be an integer. Let $G$ be a graph whose vertex set is partitioned into 
a sequence of (possibly empty) sets $B_1, B_2,\dotsc,B_m$, such that for 
every $i\in[m-k+1]$ the subgraph of $G$ induced by the set 
$\bigcup_{j=i}^{i+k-1} B_j$ has treewidth less than~$10k$. If $G$ contains an 
$r$-wall for some $r\ge 300k^{3/2}$, then there is an integer $b\le m-k$ such 
that $G$ contains $10k$ vertex-disjoint paths connecting the set $B_{\le b}$ to 
the set $B_{\ge b+k}$.
\end{lemma}
\begin{proof}
Define $q=10k$ and $s=6\big\lceil \sqrt{k} \,\big\rceil$. We will show that $G$ 
contains $s^2/2>10k$ vertex-disjoint paths satisfying the properties of the 
lemma. Let $W$ be an $r$-wall in~$G$. 

For each $i,j\in [s]$, define a $q$-wall $W_{i,j}$ as a subgraph of $W$, so that 
the walls $W_{i,j}$ are arranged roughly in an $s\times s$ lattice, every pair 
of these $q$-walls is separated by at least $2s+2$ bricks of $W$, and each 
$W_{i,j}$ is separated by at least $s^2+1$ bricks from the boundary of $W$. This 
is possible as 
$qs + (2s+1)(s-1) + 2s^2+2 \le qs +4s^2<r$.

Let $I_{i,j}$ be the set of indices $l$ such that $W_{i,j}$ intersects~$B_l$.
Since every $q$-wall contains the $q\times q$ grid as a minor, its treewidth is 
at least $q=10k$. By the assumption, the $q$-wall is not contained in any union of 
$k$ consecutive blocks~$B_l$, and in particular, the minimum and the 
maximum of $I_{i,j}$ differ by at least~$k$.

Let $M_{i,j}$ be the median of $I_{i,j}$. Let $M$ be the median of the multiset 
$\{M_{i,j}; i,j\in[s]\}$. Let $c=M - k/2$ and 
$d=M + k/2$. The numbers $c,d$ are chosen so that
every set $I_{i,j}$ with $M_{i,j}\le M$ contains a number $c_{i,j}\le c$ and 
every set $I_{i,j}$ with $M_{i,j}\ge M$ contains a number $d_{i,j}\ge d$.
Let $C,D\subseteq [s]^2$ be sets of size $s^2/2$ forming a partition of $[s]^2$ 
such that $M_{i,j}\le M$ if $(i,j)\in C$ and $M_{i,j}\ge M$ if $(i,j)\in D$. 
From every $q$-wall $W_{i,j}$ such that $(i,j)\in C$, 
choose a vertex $w_{i,j}\in B_{c_{i,j}}\cap W_{i,j}$. Similarly, from every 
$q$-wall $W_{i,j}$ such that $(i,j)\in D$, choose a vertex $w'_{i,j}\in 
B_{d_{i,j}}\cap W_{i,j}$. 

Let $w_1,w_2,\dots,w_{s^2/2}$ be a relabeling of the vertices $w_{i,j}$ 
corresponding to the lexicographic order of the pairs $(j,i)$.  
We claim that there is a relabeling $w'_1,w'_2,\dots,w'_{s^2/2}$ of the vertices 
$w'_{i,j}$ and $s^2/2$ vertex-disjoint paths $P_1, P_2, \dots, P_{s^2/2}$ in $W$ 
where each $P_t$ starts at $w_t$ and ends at $w'_t$. 
We note that with a stronger assumption on the distances between the subwalls 
$W_{i,j}$, a more general result by Robertson and 
Seymour~\cite{RS88_minorsVII_paths} would imply the existence of such paths for 
any relabeling of the vertices $w'_{i,j}$.

We choose the relabeling $w'_1,w'_2,\dots,w'_{s^2/2}$ of the vertices 
$w'_{i,j}$ 
so that it again corresponds to the lexicographic order of the pairs $(j,i)$.

We construct the paths as follows; see Figure~\ref{obr_wallpath}. Let $t\in 
[s^2/2]$ and let $i,j,i',j'\in [s]$ be such that $w_t=w_{i,j}$ and 
$w'_t=w_{i',j'}$. The path $P_t$ is composed from nine horizontal or nearly 
vertical subpaths $P_{t,1},P_{t,2},\dots,P_{t,9}$. Here by a \emph{nearly 
vertical} path in $W$ going upwards (downwards) we mean a path consisting of 
subpaths between branching vertices of $W$ with internal vertices of degree $2$, 
where the directions of the subpaths periodically alternate as up, right, up, 
left (down, right, down, left, respectively). The path $P_{t,1}$ starts at $w_t$ 
and goes up or to the right to the closest branching vertex in $W$. The path 
$P_{t,2}$ continues directly to the right to the $(2i)$th branching vertex 
outside of $W_{i,j}$. The path $P_{t,3}$ is nearly vertical continuing upwards 
to the $(2t)$th row of $W$, and it is followed by a horizontal path $P_{t,4}$ 
ending at the $(2t)$th branching vertex from the right of $W$. Then $P_{t,5}$ is 
nearly vertical 
continuing downwards, to the $(2t)$th 
row from the bottom of $W$. 

\begin{figure}
\begin{center}
\includegraphics{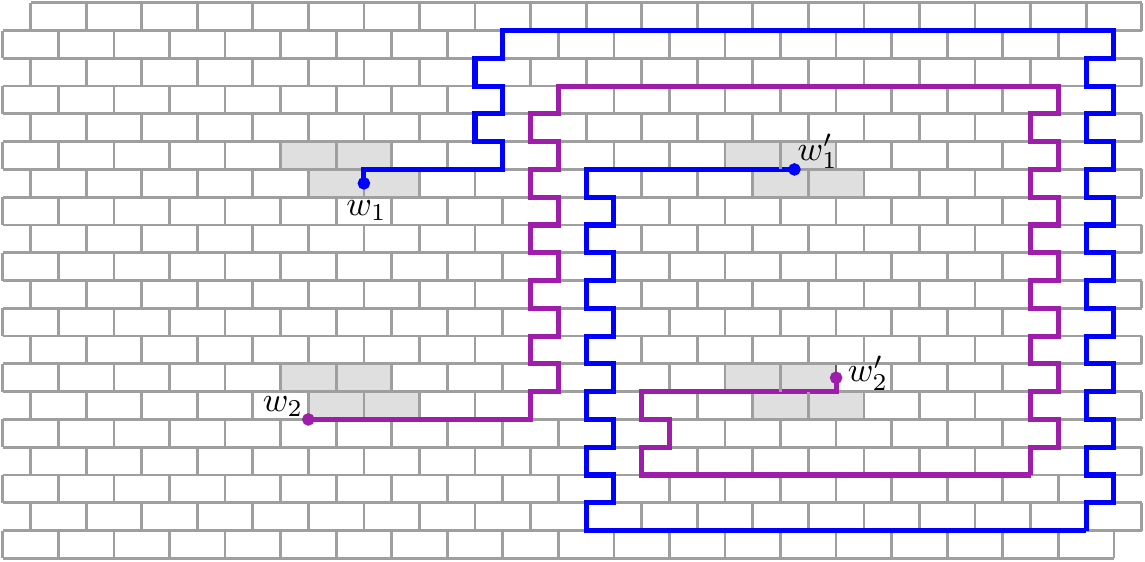}
\end{center}
\caption{Paths $P_1$ and $P_2$ for $q,s=2$.}
\label{obr_wallpath}
\end{figure}

The remaining subpaths of $P_t$ are defined symmetrically. Starting at $w'_t$, 
the path $P_{t,9}$ goes down or to the left to the closest branching vertex, 
followed by $P_{t,8}$ that goes directly to the left to the $2(s+1-i')$th 
branching vertex outside of $W_{i',j'}$. Then $P_{t,7}$ is nearly vertical 
continuing downwards to the $(2t)$th row of $W$ from the bottom, and $P_{t,6}$ 
continues directly to the right to the $(2t)$th branching vertex from the right, 
which is also the endpoint of $P_{t,5}$. 

Since all the paths connect the set $B_{\le c}$ to the set $B_{\ge d}$ with 
$d\ge c+k$, the lemma follows.
\end{proof}

The next lemma shows how to find the $k$-track in $G_\pi$, given many vertex disjoint paths between two distant blocks of the stair-decomposition.

\begin{lemma}\label{lem-good-paths}
Let $\pi\in\Av(321)$ be a permutation with a stair-decomposition 
$B_1,B_2,\dots,B_m$. Suppose that there is an odd index $a\in[m-k+1]$ such that 
the subgraph of $G_\pi$ induced by $\bigcup_{j=a}^{a+k-1} B_j$ contains $3k-2$ vertex-disjoint paths $Q_1,Q_2, \dots, Q_{3k-2}$, each of them connecting a vertex in $B_a$ to a vertex in 
$B_{a+k-1}$, and each of them containing only good edges. Then $\pi$ contains the $k$-track 
$\tau_k$, and moreover, it has an occurrence of $\tau_k$ in which the elements 
from the $j$th block of $\tau_k$ are mapped into $B_{a+j-1}$.
\end{lemma}
\begin{proof}
Truncating the paths $Q_i$ if necessary, we may assume that 
each of them has exactly one vertex in $B_a$ and exactly one vertex in $B_{a+k-1}$. Suppose 
without loss of generality that the paths are numbered in such a way that the 
endpoint of $Q_t$ in $B_a$ is to the left of the endpoint of $Q_{t+1}$ in $B_a$
for every~$t\in[3k-3]$.

For every $i,j\in[k]$, let $v_{i,j}$ be the last vertex on
$Q_{2(i-1)+j}$ that lies in $B_{a-1+j}$. Let $\sigma$ be the pattern formed by 
these $k^2$ elements $v_{i,j}$ in $\pi$. We claim that $\sigma$ is the 
$k$-track. Clearly, the sets $B_a, B_{a+1}, \dots, B_{a+k-1}$ induce a stair 
decomposition of $\sigma$, with each block having exactly $k$ elements. It 
remains to verify that consecutive blocks induce vertical or horizontal 
alternations in~$\sigma$. More precisely, we want to show that for each $j\in 
[k-1]$, we have 
\begin{equation}\label{eq_alternation}
v_{1,j}<_{j}v_{1,j+1}<_{j}v_{2,j}<_{j}v_{2,j+1}<_{j} \dots 
<_{j}v_{k,j}<_{j}v_{k,j+1}
\end{equation}
where $<_j$ is the horizontal left-to-right order on $B_{a-1+j}\cup B_{a+j}$ for 
$j$ odd and the vertical bottom-to-top order on $B_{a-1+j}\cup B_{a+j}$ for $j$ 
even.

We show \eqref{eq_alternation} only for $j=1$ since the other cases are 
analogous and follow by induction on $j$, using the induction assumption that 
$v_{1,j},v_{2,j},\dots,v_{k,j}$ is an increasing subsequence in $\pi$. 

By symmetry, it is sufficient to prove that $v_{1,1}<_{1}v_{1,2}$; in other 
words, that $v_{1,1}$ is to the left of $v_{1,2}$ in $\pi$. This will follow 
from the fact that all vertices of $Q_1$ in $B_{a+1}$ are to the left of all 
vertices of $Q_2$ in $B_{a+1}$. Indeed, the neighbor $v'$ of $v_{1,1}$ in $Q_1$ 
is at a position adjacent to $v_{1,1}$ in $\pi$, so every element in $B_{a+1}$ 
that is to the right of $v'$ is also to the right of $v_{1,1}$.

Let $G$ be the subgraph of $G_{\pi}$ consisting of the vertices $B_a \cup 
B_{a+1} \cup \dots \cup B_{a+k-1}$ and all the good edges among these vertices. 
Clearly, $G$ contains all the paths $Q_1,Q_2, \dots, Q_{3k-2}$. We observe that 
$G$ has the following planar drawing; see Figure~\ref{obr_planar}. For each 
$j\in [k]$, draw the vertices of $B_{a-1+j}$ on the vertical line with 
$x$-coordinate $j$, ordered from bottom to top according to the order $<_{j}$. 
Then draw all the edges as straight-line segments.

\begin{figure}
\begin{center}
\includegraphics{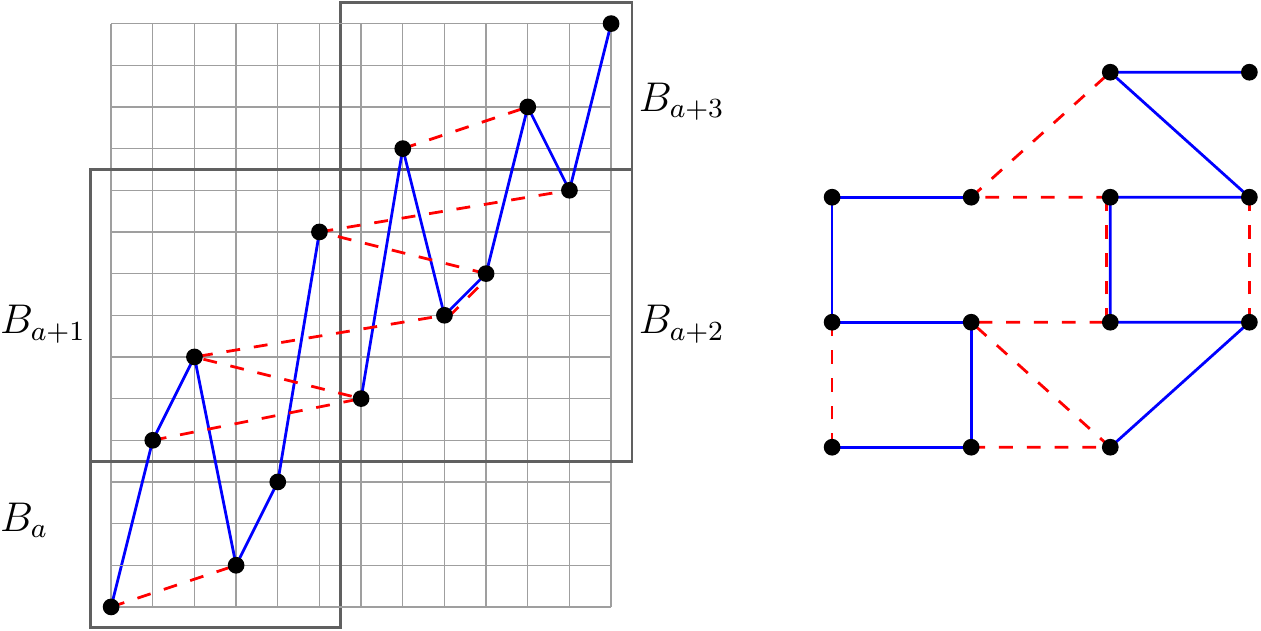}
\end{center}
\caption{Left: a subgraph $G$ of $G_{\pi}$ consisting of good edges. Right: a 
planar drawing of $G$.}
\label{obr_planar}
\end{figure}

Let $\Omega=\{(x,y)\in \mathbb{R}^2; 1\le x \le k\}$, which is a convex region 
containing all the vertices and edges of the drawing.
All the vertices of $B_{a}$ and $B_{a+k-1}$ are on the boundary of $\Omega$.
Each of the paths $Q_t$ connects a vertex in $B_{a}$ with a vertex in 
$B_{a+k-1}$, and all the inner vertices of $Q_t$ are in $B_{a+1} \cup B_{a+2} \cup \dots \cup B_{a+k-2}$. 
Therefore, each path $Q_t$ divides $\Omega$ into two regions: the region 
\emph{below} $Q_t$ and the region \emph{above} $Q_t$. In particular, $Q_2$ is 
contained in the region above $Q_1$.

\begin{figure}
\begin{center}
\includegraphics{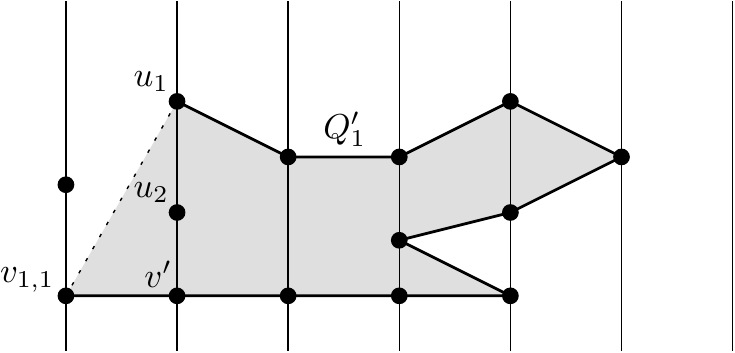}
\end{center}
\caption{If $u_2$ is below $u_1$, then $Q_2$ is forced to return to the left 
boundary of $\Omega$.}
\label{obr_Q1_Q2}
\end{figure}

Suppose for contradiction that there are vertices $u_1\in Q_1 \cup B_{a+1}$ and 
$u_2\in Q_2 \cup B_{a+1}$ such that $u_2$ is below $u_1$ on the line with 
$x$-coordinate $2$; see Figure~\ref{obr_Q1_Q2}.
Let $Q'_1$ be the subpath of $Q_1$ between $v_{1,1}$ and $u_1$.  Since $u_2$ is 
in the region above $Q_1$, the vertex $v'$, which is the neighbor of $v_{1,1}$ 
on $Q_1$, is below $u_2$. Consequently, $u_2$ is inside the bounded region 
formed by $Q'_1$ and the segment $u_1v_{1,1}$. Since $Q_1$ and $Q_2$ do not 
cross, it follows that $Q_2$ crosses the segment $u_1v_{1,1}$ twice. But this 
means that there are at least two vertices on $Q_2$ in $B_a$, which is a 
contradiction.
\end{proof}

\begin{proof}[Proof of Theorem~\ref{theorem_hledani_ktracku}]
We shall prove that the theorem holds for the function $g(k)=300k^{3/2}$. 
Suppose that $\pi\in\Av(321)$ is a permutation such that $G_\pi$ contains a 
$g(k)$-wall. Let $B_1,\allowbreak B_2,\dots,\allowbreak B_m$ be a stair-decomposition of~$\pi$. By 
Lemma~\ref{lemma_tw_blocks}, the subgraph of $G_\pi$ induced by any $k$ 
consecutive blocks of the decomposition has treewidth at most~$2k$. We may 
therefore apply Lemma~\ref{lem-graf} to the graph $G_\pi$ and the sets 
$B_1,B_2,\dots,B_m$, to conclude that there is an index $b\in[m-k]$ and a set of 
vertex-disjoint paths $P_1, P_2,\dots,P_{10k}$, each connecting $B_{\le b}$ to 
$B_{\ge b+k}$.

Let $a$ be the smallest odd integer greater than or equal to~$b$. Note that 
$G_\pi$ has at most two edges connecting a vertex in $B_{<a}$ to a vertex in 
$B_{>a}$, since $B_{>a}$ is above and to the right of~$B_{<a}$. By the same 
argument, there are at most two edges between $B_{<a+k-1}$ and $B_{>a+k-1}$.
We remove from  $P_1, P_2,\dots,P_{10k}$ all the paths using any of these  
edges, and observe that each of the remaining paths contains a vertex from $B_a$ 
as well as from $B_{a+k-1}$. Truncating these paths if necessary, we may assume 
that all their vertices are in the set $\bigcup_{j=a}^{a+k-1}B_j$. By 
Lemma~\ref{lemma_bad_edges}, this set induces at most $k-1$ bad edges in 
$G_\pi$. Removing from our set of paths any path containing a bad edge, we are 
left with at least $10k-4-k+1>3k-2$ paths from $B_a$ to $B_{a+k-1}$ which only 
contain good edges. By Lemma~\ref{lem-good-paths}, $\pi$ contains the $k$-track, 
as claimed.
\end{proof}

\subsection{Proof of Theorem~\ref{veta_polynomialni}}
Let $\cC$ be a proper subclass of $\C{2}{0}$. Let $s$ be the smallest positive integer such that there is a $321$-avoiding $s$-permutation that is not contained in $\cC$. By Lemma~\ref{lemma_ktrack}, the $s$-track is not contained in $\cC$ either. By Theorem~\ref{theorem_hledani_ktracku}, for some function $g(s)\le O(s^{3/2})$, there is no permutation $\pi\in\cC$ such that $G_{\pi}$ contains a $g(s)$-wall. Consequently, no such graph $G_{\pi}$ contains the $(2g(s)+2) \times (2g(s)+2)$ grid as a minor. By Theorem~\ref{veta_grid}, for every permutation $\pi\in\cC$, the treewidth of the graph $G_{\pi}$ is at most $f(2g(s)+2))\le s^{28.5}(\log{s})^{O(1)}$. Finally, Theorem~\ref{veta_o_omezene_sirce} implies that for $\pi\in S_k\cap \cC$ and $\tau\in S_n$, the problem whether $\pi$ is contained in $\tau$ can be solved in time $kn^{s^{28.5}(\log{s})^{O(1)}}$.

\section{Skew-merged patterns}

Recall that $\Cplus$ is the class $\Av(2143,3412,3142)$. Our goal is to prove 
Theorems~\ref{thm-skew} and~\ref{veta_polynomialni_skew}. To this end, we will 
first show that permutations from $\Cplus$ admit a decomposition which is 
analogous to the stair-decomposition for the class $\C 2 0$. This will 
allow us to show that the proofs of Theorems~\ref{thm-main} 
and~\ref{veta_polynomialni} can be straightforwardly adapted into proofs of 
Theorems~\ref{thm-skew} and~\ref{veta_polynomialni_skew}, respectively.

\begin{definition}\label{def-spiral}
Let $\pi\in S_k$ be a permutation. A \emph{spiral decomposition} of $\pi$ is a 
partition of $\pi$ into a sequence $B_1, B_2, \dotsc, B_m$, where each $B_i$ is 
a possibly empty subset of elements of~$\pi$, satisfying the 
following properties:
\begin{itemize}
 \item[(a)] If $i$ is odd then $B_i$ is a decreasing subsequence, and if $i$ is 
even then $B_i$ is an increasing subsequence of~$\pi$.
 \item[(b)] For $i\in[m]$, let $r_i\in\{0,1,2,3\}$ 
be the remainder of $i$ modulo~4. See the right part of 
Figure~\ref{fig-twirl}. If $r_i=0$ then $B_{i}$ is above 
$B_{i-1}$ and $B_{>i}$ is above and to the left of $B_{i-1}$. If $r_i=1$ 
and $i>1$ 
then $B_{i}$ is to the left of $B_{i-1}$ and $B_{>i}$ is below and to the 
left of~$B_{i-1}$. If $r_i=2$ then $B_{i}$ is below $B_{i-1}$ and 
$B_{>i}$ is below and to the right of~$B_{i-1}$. If $r_i=3$ then $B_{i}$ is to 
the right of $B_{i-1}$ and $B_{>i}$ is above and to the right of~$B_{i-1}$.
\end{itemize}

The subsequences $B_1,B_2,\dotsc,B_m$ are the \emph{blocks} of the spiral 
decomposition.
\end{definition}

Note that if a block $B_i$ is empty, then $B_i$ is simultaneously to the left, 
to the right, above and below the block $B_{i-1}$.

\begin{lemma}\label{lem-spiral}
A permutation $\pi$ belongs to the class $\Cplus$ if and only if $\pi$ has a 
spiral decomposition. Moreover, every permutation in $\Cplus$ has a spiral 
decomposition with no four consecutive empty blocks.
\end{lemma}
\begin{proof}
Suppose first that $\pi$ is in $\Cplus$. In particular, $\pi$ is skew-merged, 
and therefore it can be partitioned into an increasing subsequence $I$ and a 
decreasing subsequence~$D$. 

We now define an infinite sequence $B_1, B_2,\dotsc$, where each $B_i$ is a 
possibly empty subset of~$\pi$. The $B_i$ will be defined inductively. For 
an $i\in\mathbb{N}$, suppose that $B_1, B_2,\dotsc, B_{i-1}$ have already been 
defined. Define $I_i$ as $I\setminus (B_1\cup B_2\cup\dots\cup B_{i-1})$ and 
$D_i$ as $D\setminus (B_1\cup B_2\cup\dots\cup B_{i-1})$. Let 
$r_i\in\{0,1,2,3\}$ be the remainder of $i$ modulo 4. To define $B_i$, we 
distinguish four cases:
\begin{itemize}
\item If $r_i=0$, then $B_i=\{x\in I_i;\; x\text{ is to the right of }D_i\}$. 
\item If $r_i=1$, then $B_i=\{x\in D_i;\; x\text{ is above }I_i\}$.
\item If $r_i=2$, then $B_i=\{x\in I_i;\; x\text{ is to the left of }D_i\}$.
\item If $r_i=3$, then $B_i=\{x\in D_i;\; x\text{ is below }I_i\}$.
\end{itemize}

Let $B_m$ be the last nonempty block in the sequence $B_1, B_2,\dotsc$ defined 
above (if all the blocks are empty, set $m=0$). We now show that 
$B_1,\allowbreak B_2,\dotsc,\allowbreak B_m$ is a spiral decomposition of~$\pi$. 
Clearly, $B_i$ is a 
decreasing sequence for odd $i$ and an increasing sequence for even~$i$.
Let $B_\infty$ be the set of elements of $\pi$ not belonging to $B_i$ 
for any integer~$i$.

Let us verify that the blocks have the correct mutual position. Let $B_{>i}$ 
be the set of elements of $\pi$ not belonging to $B_1\cup B_2\cup 
\dots\cup B_i$; in particular, $B_\infty$ is a subset of $B_{>i}$. We claim 
that for any $i\in \{2,3,\dotsc,m\}$, the mutual positions of $B_{i-1}$, $B_i$ 
and 
$B_{>i}$ are the same as in part (b) of Definition~\ref{def-spiral}. 

Suppose that $r_i=0$, the other three cases being analogous. Note that 
$I_i=I_{i-1}$ and $D_i=D_{i-1}\setminus B_{i-1}$. Since $B_{i-1}$ is below 
$I_{i-1}$ by definition and $B_i$ is a subset of $I_{i-1}$, we see that 
$B_i$ above $B_{i-1}$.

To see that $B_{>i}$ is to the left and above $B_{i-1}$, notice that $B_{>i}$ 
is the disjoint union of $D_i$ and $I_i\setminus B_i$. Clearly $D_i$ is above 
and to the left of $B_{i-1}$, and $I_i$ is above $B_{i-1}$ by 
definition of $B_{i-1}$. Moreover, every element of $I_i\setminus B_i$ must be 
to the left of at least one element of $D_i$ by definition of $B_i$, and 
therefore also to the left of~$B_{i-1}$. This shows that $B_{>i}$ is to the left 
and above~$B_{i-1}$.

To prove that $B_1,B_2,\dotsc,B_m$ is a spiral decomposition, it remains to show 
that each element of $\pi$ is in $B_1\cup B_2\cup\dots\cup B_m$. Suppose for 
contradiction that this is not the case, that is, suppose that $B_\infty$ is 
not empty. Let $b,l,t,r$ be the bottommost, leftmost, topmost and rightmost
element of $B_\infty$, respectively. 

Let $M$ be the smallest multiple of $4$ larger than~$m$. Note that 
$B_{>M}=B_{>m}=B_\infty$, since all the blocks following $B_m$ are empty. 
Consider the 
four empty blocks $B_M$, $B_{M+1}$, $B_{M+2}$, and $B_{M+3}$. We 
may deduce that $r$ belongs to $D_M$, for otherwise it would be to the right of 
$D_M$ and would belong to~$B_M$. By the same argument, $t$ belongs to $I_M$, 
$l$ belongs to $D_M$, and $b$ belongs to~$I_M$. It follows that $l,b,t,r$ form 
an occurrence of the pattern $3142$, contradicting the assumption that $\pi$ is 
in~$\Cplus$.

We have thus shown that each permutation in $\Cplus$ has a spiral 
decomposition. 
The argument actually shows that we may find a spiral decomposition with no 
four 
consecutive empty blocks.

We now show that every permutation with a spiral decomposition is in~$\Cplus$. 
Let $\pi$ be a permutation with a spiral decomposition $B_1,B_2,\dotsc,B_m$. It 
easily follows from Definition~\ref{def-spiral} that the union of the 
odd-numbered blocks is a decreasing sequence and the union of the even-numbered 
ones is an increasing sequence. In particular, $\pi$ is skew-merged.

It remains to show that $\pi$ avoids $3142$. Suppose for contradiction that 
$\pi$ has a subsequence $S=(l,b,t,r)$ forming an occurrence of $3142$. Since 
there is only one way to partition $3142$ into a decreasing and an increasing 
subsequence, we know that $l$ and $r$ belong to odd-numbered blocks, and $t$ 
and $b$ to even-numbered ones. 

Let $B_j$ be the first block of the spiral decomposition containing at least 
one element of~$S$. Suppose that $j$ is even, the other case being analogous. 
Then $B_j$ contains $t$ or~$b$, and both $l$ and $r$ belong to~$B_{>j}$. 
From part (b) of Definition~\ref{def-spiral} we may deduce that 
either $B_{>j}$ is to the left of $B_j$ (if $j \equiv 0 \mod 4$), or $B_{>j}$ 
is to the right of $B_j$ (if $j\equiv 2 \mod 4$). This contradicts the fact 
that both $l$ and $r$ belong to~$B_{>j}$ and at least one of $t$ and $b$ is 
in~$B_j$.
\end{proof}

\subsection{The twirl}

We now introduce an operation called the twirl, whose main 
purpose is to transform a permutation $\pi$ with a given stair-decomposition 
$B_1,B_2,\dotsc,B_m$ into a permutation $\pi^*$ with a spiral decomposition 
$B^*_1,B^*_2,\dotsc, B^*_m$. Intuitively, each block $B^*_i$ will be obtained 
by 
rotating or reflecting $B_i$, and then rearranging the blocks into a spiral.

\begin{figure}
\begin{center}
\includegraphics{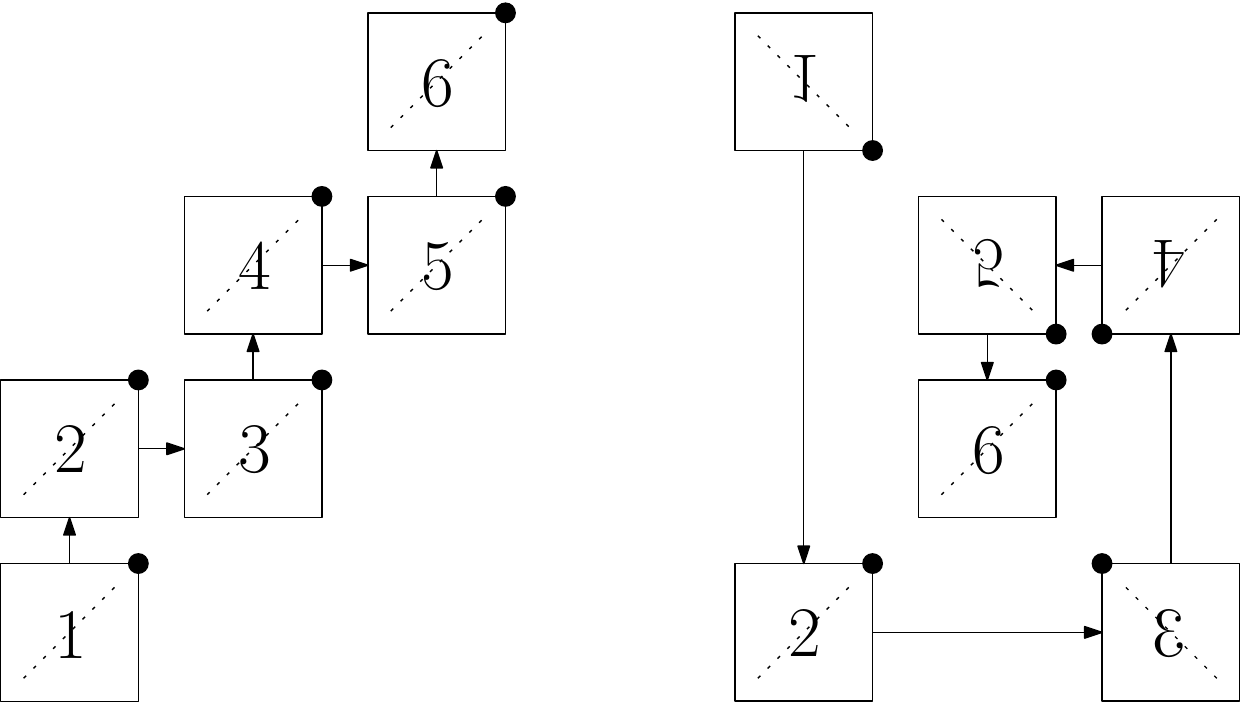}
\end{center}
 \caption{Left: a permutation $\pi\in\C{2}{0}$ with a stair-decomposition into 
six blocks. The top-right corner of each block is highlighted. Right: the 
result of twirling $\pi$---a spiral decomposition whose blocks are rotated and 
reflected versions of the blocks of~$\pi$, with the highlighted corners now 
pointing towards the center of the spiral.}\label{fig-twirl}
\end{figure}

Formally, given a permutation $\pi$ with a stair-decomposition $B_1,\allowbreak 
B_2,\dotsc,\allowbreak B_m$, 
the \emph{twirl} of $\pi$ (with respect to $B_1,B_2,\dotsc,B_m$) is a 
permutation $\pi^*$ with a spiral decomposition $B^*_1,B^*_2,\dotsc,B^*_m$, 
determined 
by these rules:
\begin{itemize}
 \item For each $i$, the block $B_i^*$ is obtained by rotating or reflecting 
$B_i$, in such a way that rows of $B_i$ get mapped to rows of $B_i^*$ and 
similarly for columns. Moreover, the top-right corner of $B_i$ will be mapped 
to the corner of $B^*_i$ that is nearest to the center of the spiral; see 
Figure~\ref{fig-twirl}.
\item For every pair of consecutive blocks $B_i$ and $B_{i+1}$, let 
$x\in B_i$ and $y\in B_{i+1}$ be a pair of elements, and let $x^*$ and $y^*$  
be the corresponding elements of $B^*_i$ and $B^*_{i+1}$, respectively. Then 
the mutual position of $x^*$ and $y^*$ corresponds to the mutual position of 
$x$ and $y$, in the following sense:
\begin{itemize}
\item if $i\equiv 0 \mod 4$, then $x^*$ is above $y^*$ if and only if $x$ is 
below~$y$,
\item if $i\equiv 1 \mod 4$, then $x^*$ is to the right of $y^*$ if and only if 
$x$ is to the right of~$y$,
\item if $i\equiv 2 \mod 4$, then $x^*$ is above $y^*$ if and only 
if $x$ is above~$y$,
\item if $i\equiv 3 \mod 4$, then $x^*$ is to the right of $y^*$ if and only if 
$x$ is left of~$y$.
\end{itemize}
\end{itemize}
Note that the properties above, together with the fact that 
$B^*_1,B^*_2,\dotsc,B^*_m$ form a spiral decomposition, determine $\pi^*$ 
uniquely.
We say that $\pi^*$ is the \emph{twirl} of $\pi$ and $\pi$ is the 
\emph{untwirl} of $\pi^*$ (with respect to given decompositions 
$B_1,B_2,\dotsc,B_m$ and $B^*_1,B^*_2,\dotsc,B^*_m$, respectively).

We will also need to apply this transform in situations when the blocks of the 
decomposition do not necessarily induce a monotone subsequence. We therefore 
generalize the two notions of decomposition appropriately. We say that a 
partition of a permutation $\pi$ into a sequence of blocks $B_1,B_2,\dotsc, 
B_m$ 
forms 
a \emph{relaxed stair-decomposition} if it satisfies all the properties of 
stair-decomposition except possibly for the property that each $B_i$ forms an 
increasing subsequence. Similarly, a \emph{relaxed spiral decomposition} 
satisfies all the properties of a spiral decomposition, except perhaps the 
property that each block is a monotone subsequence. 

The twirl transform can be applied to a permutation with a relaxed 
stair-decomposition to obtain a permutation with a relaxed spiral 
decomposition.

\subsection{Proof of Theorem~\ref{thm-skew}}
We want to show that PPM is NP-hard for instances where the pattern belongs to 
$\Cplus$ and the text to~$\C 2 1$. Our idea is to use the pattern $\pi\in\C 2 
0$ and the text $\tau\in\C 3 0$ constructed from a 3-SAT formula $\Phi$ in the 
proof of Theorem~\ref{thm-main}, and twirl them into a pattern $\pi^*\in\C 1 1$ 
and text $\tau^*\in\C 2 1$. We will then show that $\Phi$ is satisfiable if and 
only if $\tau^*$ contains~$\pi^*$.

For the twirl to be well defined, we need to specify a stair-decomposition of 
$\pi$ and a relaxed stair-decomposition of~$\tau$. Recall that $\pi$ is a 
disjoint union of an increasing sequence $A$ (the anchor) and a $v$-fold 
staircase $X$ with $2c+1$ steps. The staircase $X$ itself is a sequence of 
bends 
$Q_1,\allowbreak P_1,\allowbreak Q_2,\allowbreak P_2,\dotsc,\allowbreak 
Q_{2c+1},\allowbreak P_{2c+1}$, with $Q_i$ being the $i$th outer bend and 
$P_i$ the $i$th inner bend of~$X$. The sequence $A, Q_1, P_1, Q_2, P_2,\dotsc, 
Q_{2c+1}, P_{2c+1}$ is clearly a stair-decomposition of $\pi$, and we 
obtain $\pi^*$ by twirling $\pi$ with respect to this decomposition.

For $\tau$, we consider the relaxed 
stair-decomposition whose blocks are the anchor $A'$, followed by the 
modified bends 
$\moQ_{1},\moP_{1},\moQ_{2},\moP_{2},\dotsc,\moQ_{2c+1},\moP_{2c+1}$ of the 
modified staircase~$\mo{Y}$. Let $\tau^*$ be the permutation obtained by 
twirling~$\tau$.

Note that the union of the odd-numbered blocks of $\tau^*$ is a decreasing 
sequence while the union of the even-numbered blocks can be decomposed into at 
most two increasing sequences. It follows that $\tau^*$ is in~$\C 2 1$.

We easily check that $\tau^*$ contains $\pi^*$ if and only if $\Phi$ is 
satisfiable, by following the same reasoning as in the proof of 
Theorem~\ref{thm-main}. This completes the proof of Theorem~\ref{thm-skew}.

\subsection{Proof of Theorem~\ref{thm-principal}}
Let $F$ denote the set of permutations $\{1,12,21,132,213,231,312\}$.
Recall that Theorem~\ref{thm-principal} states that $\Av(\alpha)$-Pattern 
PPM is polynomial for each $\alpha$ in $F$ and NP-complete for any 
other~$\alpha$.

First of all, we note that for each $\alpha\in F$, the class $\Av(\alpha)$ is a 
subclass of the class $\Av(2413,3142)$ of separable permutations. Since 
$\Av(2413,3142)$-Pattern PPM is polynomial by a result of Bose, Buss and 
Lubiw~\cite{BBL}, $\Av(\alpha)$-Pattern PPM is polynomial as well.

On the other hand, if $\alpha$ contains $321$ as a subpermutation, then 
$\Av(\alpha)$ contains $\Av(321)$ as a subclass. Since $\Av(321)$-Pattern PPM 
is NP-hard by Theorem~\ref{thm-main}, $\Av(\alpha)$-Pattern PPM is NP-hard as 
well. By symmetry, $\Av(\alpha)$-Pattern PPM is also hard when $\alpha$ 
contains~$123$.

There are only four permutations that do not belong to $F$ and do not contain a 
monotone subsequence of length $3$, namely $2143$, $3412$, $2413$, and $3142$. 
Theorem~\ref{thm-skew} implies that $\Av(\alpha)$-Pattern PPM is hard for any 
$\alpha\in\{2143, 3412, 3142\}$. By symmetry, the problem is hard for 
$\alpha=2413$ as well. This completes the proof of Theorem~\ref{thm-principal}.

\subsection{Proof of Theorem~\ref{veta_polynomialni_skew}}

Our goal is to show that for every proper subclass $\cC$ of $\Cplus$, the 
$\cC$-Pattern PPM problem is polynomial. We will again use the twirl to adapt 
the proof of Theorem~\ref{veta_polynomialni} that we gave in 
Section~\ref{sec-poly}. 

As in the proof of Theorem~\ref{veta_polynomialni}, our goal is to show that 
for permutations $\pi$ taken from any proper subclass of $\Cplus$, the graphs 
$G_\pi$ have bounded treewidth, or more precisely, that every permutation 
$\pi\in\Cplus$ whose graph $G_\pi$ has sufficiently large treewidth contains a 
universal pattern.

Let $\pi$ be a permutation with a stair-decomposition $B_1,B_2,\dotsc,B_m$, and 
let $\sigma$ be a permutation with a stair-decomposition $C_1,C_2,\dotsc,C_k$ 
such 
that $k\le m$. We say that $\sigma$ has a \emph{block-preserving occurrence} in 
$\pi$ (with respect to the decompositions $B_1,B_2,\dotsc,B_m$ and 
$C_1,C_2,\dotsc,C_k$) 
if it has an occurrence in $\pi$ with the property that for each $i\le k$ the 
elements of $C_i$ are mapped to the elements of~$B_i$. We also use the same 
terminology when dealing with spiral decompositions.

Notice that block-preserving occurrences are preserved by the twirl operation, 
as formalized by the next observation.

\begin{observation}\label{obs-twirl}
Let $\pi$ be a permutation with a stair-decomposition $B_1,\allowbreak 
B_2,\dotsc,\allowbreak 
B_m$, and let $\sigma$ be a permutation with a stair-decomposition 
$C_1,\allowbreak C_2,\dotsc,\allowbreak C_k$. Let $\pi^*$ with spiral 
decomposition $B^*_1,B^*_2,\dotsc,B^*_m$ 
and $\sigma^*$ with spiral decomposition $C^*_1,C^*_2,\dotsc,C^*_k$ be obtained 
by 
twirling $\pi$ and $\sigma$, respectively. Then $\sigma$ has a block-preserving 
occurrence in $\pi$ if and only if $\sigma^*$ has a block-preserving occurrence 
in~$\pi^*$.
\end{observation}

Let $\tau_k$ be the $k$-track permutation, and let $A_1,A_2,\dotsc,A_k$ be its 
stair-decom\-po\-sition into blocks of size $k$, introduced in 
Section~\ref{sec-poly}. The \emph{$k$-spiral} is the permutation $\tau^*_k$ 
obtained by twirling $\tau_k$ with respect to this stair-decomposition. When 
dealing with $k$-spirals, we always assume they have a spiral decomposition 
obtained by twirling the stair decomposition $A_1,\allowbreak 
A_2,\dotsc,\allowbreak A_k$.

\begin{lemma}\label{lem-qspiral}
Every $k$-permutation $\pi\in\Cplus$ with a spiral decomposition into $m$ blocks
 has a block-preserving occurrence in the $q$-spiral, for $q=\max\{k,m\}$.
\end{lemma}
\begin{proof}
 This follows directly from Lemma~\ref{lemma_ktrack}, via 
Observation~\ref{obs-twirl}.
\end{proof}
Since each $n$-permutation in $\Cplus$ has a spiral decomposition with at most 
$4n$ blocks, we conclude that each $n$-permutation in $\Cplus$ is contained in 
the $4n$-spiral.

For a permutation $\pi=\pi(1),\pi(2),\dotsc,\pi(n)$, the 
\emph{reverse-complement} of 
$\pi$ is the permutation $n+1-\pi(n), n+1-\pi(n-1),\dotsc,n+1-\pi(1)$. 
Intuitively, the diagram of the reverse-complement is obtained from the diagram 
of $\pi$ by a rotation of 180 degrees. Notice that $\Cplus$ is closed under 
reverse-complements. Consequently, each $n$-permutation 
in $\Cplus$ is contained in the reverse-complement of the $4n$-spiral.

The notion of good edges and bad edges in $G_\pi$, which we introduced in 
Section~\ref{sec-poly}, can be extended to spiral decompositions.
Let $\pi\in\Cplus$ be a permutation with a spiral decomposition 
$B_1,B_2,\dotsc,B_m$. 
Let $xy$ be an edge of $G_\pi$, such that $x\in B_i$ and $y\in B_j$, with $i\le 
j$. Then the edge $xy$ is \emph{good} if it satisfies at least one of the 
following properties:
\begin{itemize}
\item $i=j$
\item $xy$ is blue, $i$ is odd and $j=i+1$, or
\item $xy$ is red, $i$ is even and $j=i+1$.
\end{itemize}
Otherwise $xy$ is \emph{bad}.

The twirl operation clearly preserves the good edges.
\begin{observation}\label{obs-good}
 Let $\pi$ be a permutation with a stair-decomposition 
$B_1,\allowbreak B_2,\dotsc,\allowbreak B_m$ and 
let $\pi^*$ be the permutation with spiral decomposition  $B_1^*,\allowbreak 
B_2^*,\dotsc,\allowbreak B_m^*$
obtained by 
twirling~$\pi$. Let $x$ and $y$ be two elements of $\pi$ and $x^*$ and $y^*$ 
the corresponding elements of~$\pi^*$. Then $xy$ is a good edge of $G_\pi$ if 
and only if $x^* y^*$ is a good edge of $G_{\pi^*}$.
\end{observation}

\begin{lemma}\label{lem-spiral-bad}
If a permutation $\pi\in\Cplus$ has a spiral decomposition with $m$ blocks, 
then $G_\pi$ has at most $4m$ bad edges.
\end{lemma}
\begin{proof}
 It is easy to see that if $xy$ is a bad edge in $G_\pi$, then both $x$ and $y$ 
are either leftmost or rightmost elements of their blocks in the spiral 
decomposition of~$\pi$. Since every vertex of $G_\pi$ has degree at most 4 
and at most $2m$ vertices are incident to bad edges, $G_\pi$ has at most 
$4m$ bad edges.
\end{proof}

\begin{lemma}\label{lem-spiral-tw}
If $\pi\in\Cplus$ has a spiral decomposition with $m$ blocks, 
then the treewidth of $G_\pi$ is at most~$6m$.
\end{lemma}
\begin{proof}
Untwirl $\pi$ into a permutation $\pi^*\in\Av(321)$. By 
Lemma~\ref{lemma_tw_blocks}, the graph $G_{\pi^*}$ has treewidth at most~$2m$. 
By Observation~\ref{obs-good}, every good edge of $G_\pi$ is also an edge of 
$G_{\pi^*}$. Therefore, by Lemma~\ref{lem-spiral-bad}, $G_\pi$ can be obtained 
from $G_{\pi^*}$ by inserting at most $4m$ new bad edges, and possibly also 
removing some bad edges of~$G_{\pi^*}$. Inserting an edge into a graph may 
increase its treewidth 
by at most $1$, and removing an edge cannot increase the treewidth. Therefore, 
the treewidth of $G_\pi$ is at most $6m$, as claimed. 
\end{proof}

To prove Theorem~\ref{veta_polynomialni_skew}, we first prove the following 
claim, which is the $\Cplus$-analogue of Theorem~\ref{theorem_hledani_ktracku}, 
and is proved by a similar argument.

\begin{theorem}\label{thm-hledani-k-spiral}
 Let $\pi\in\Cplus$ be a permutation such that $G_\pi$ contains a $g(k)$-wall, 
where $g(k)=300k^{3/2}$. Then $\pi$ contains the $k$-spiral or the 
reverse-complement of the $k$-spiral.
\end{theorem}
\begin{proof}
Let $\pi\in\Cplus$ have a spiral decomposition $B_1,B_2,\dotsc,B_m$.
 By Lemma~\ref{lem-spiral-tw}, the subgraph of $G_\pi$ induced by $k$ 
consecutive blocks of the decomposition has treewidth at most~$6k$. We 
may invoke Lemma~\ref{lem-graf}, to obtain $10k$ vertex-disjoint paths from 
$B_{\le b}$ to $B_{\ge b+k}$ for some~$b$. Let $a$ be the smallest odd 
integer greater than or equal to~$b$. By discarding at most four paths 
containing an edge from $B_{<a}$ to $B_{>a}$ or an edge from $B_{<a+k-1}$ to 
$B_{>a+k-1}$, truncating the remaining paths to only contain vertices from 
$\bigcup_{j=a}^{a+k-1}B_j$, and discarding at most $4k$ paths containing bad 
edges, we end up with at least $10k-4-4k=6k-4>3k-2$ paths from $B_a$ to 
$B_{a+k-1}$ which only contain good edges.

Untwirl $\pi$ into a permutation $\pi^*\in\Av(321)$ with a stair-decomposition 
$B^*_1,B^*_2,\dotsc, B^*_m$. Since the untwirl preserves good edges, we know by 
Lemma~\ref{lem-good-paths} that $\pi^*$ contains the $k$-track $\tau_k$, and 
the occurrence of $\tau_k$ is block-preserving with respect to the 
stair-decomposition in which the first $a-1$ blocks of $\tau_k$ are empty and 
the remaining blocks have size~$k$. By Observation~\ref{obs-twirl}, we 
conclude that $\pi$ contains the $k$-spiral (if $a\equiv 1\mod 4$) or the 
reverse-complement of the $k$-spiral (if $a\equiv 3 \mod 4$).
\end{proof}

Theorem~\ref{thm-hledani-k-spiral} implies Theorem~\ref{veta_polynomialni_skew} 
in exactly the same way as Theorem~\ref{theorem_hledani_ktracku} implies 
Theorem~\ref{veta_polynomialni}. Theorem~\ref{veta_polynomialni_skew} is thus 
proved.

\section{Final remarks and open problems}
We have shown that for every proper subclass $\cC$ of $\Av(321)$, the 
$\cC$-Pattern PPM problem is polynomial-time solvable, and the same is true for 
proper subclasses of $\Cplus=\Av(2143,\allowbreak 3412,\allowbreak 3142)$. The 
proofs are based on the 
fact that for permutations in such a class $\cC$, the associated graph 
$G_\pi$ has bounded treewidth, which, by results of Ahal and 
Rabinovich~\cite{AR08_subpattern}, implies that $\cC$-Pattern PPM is 
polynomial. 

In fact, in all the cases when $\cC$-Pattern PPM is known to be polynomial, the 
class $\cC$ has bounded treewidth of~$G_\pi$. We may therefore ask whether in 
fact the boundedness of treewidth can precisely characterize the polynomial 
cases of $\cC$-Pattern PPM.

\begin{problem}
 Is the $\cC$-Pattern PPM problem NP-hard for every class $\cC$ for which the 
treewidth of~$G_\pi$ is unbounded?
\end{problem}

For the closely related decision problem $\cC$-PPM, we have obtained 
NP-hardness when $\cC$ is the class $\C{3}{0}=\Av(4321)$. This gives the first 
known example of a 
proper permutation class $\cC$ for which the problem is hard. We have 
subsequently shown that $\C{2}{1}$-PPM is NP-hard as well.
It follows that $\Av(\rho)$-PPM is NP-hard for any permutation $\rho$ not 
belonging to the finite set $\C{3}{0}\cap\C03\cap\C21\cap\C12$, and in 
particular, 
$\Av(\rho)$-PPM is hard for every $\rho$ of size at least~$10$. 
However, a precise characterization of the polynomial cases of $\cC$-PPM is 
currently out of our reach, even for principal classes~$\cC$.

\begin{problem}
 For which permutations $\rho$ can $\Av(\rho)$-PPM be solved in polynomial time?
\end{problem}

\section*{Acknowledgements}
The authors wish to thank Ida Kantor for many valuable discussions.

\end{document}